\documentclass[11pt,reqno]{amsart}
\usepackage[mathscr]{eucal}
\usepackage{mathrsfs}
\usepackage{amsfonts}
\usepackage{amsmath}
\usepackage{amsthm}
\usepackage{amssymb}
\usepackage{latexsym}
\usepackage[all]{xy}
\usepackage{array}
\usepackage{graphicx}
\usepackage[dvipsnames]{xcolor}
\newcommand{\red}[1]{{\color{red}#1}}
\newcommand{\blue}[1]{{\color{blue}#1}}

\usepackage{mathtools}

\usepackage{latexsym}
\usepackage[mathscr]{eucal}
\usepackage{palatino}

\usepackage[bookmarks=false]{hyperref}
\usepackage{hyperref}
\usepackage{xr-hyper}

\usepackage{graphicx}

\theoremstyle{plain}
\newtheorem{theorem}[equation]{Theorem}
\newtheorem{corollary}[equation]{Corollary}
\newtheorem{proposition}[equation]{Proposition}
\newtheorem{lemma}[equation]{Lemma}

\theoremstyle{definition}

\newtheorem*{definition}{Definition}

\newtheorem{rem}[equation]{Remark}

\newtheorem{claim}[equation]{Claim}

\makeatletter
\renewcommand{\subsection}{\@startsection{subsection}{2}{0pt}{-3ex
plus -1ex minus -0.2ex}{-2mm plus -0pt minus
-2pt}{\normalfont\bfseries}} \makeatother

\numberwithin{equation}{subsection}
\allowdisplaybreaks

\newlength{\dhatheight}

\DeclareMathOperator{\id}{\mathrm{Id}}
\DeclareMathOperator{\mmod}{ \operatorname{\textsl{mod}} }
\DeclareMathOperator{\supp}{\mathrm{Supp}}

\DeclareMathOperator{\Ker}{\mathrm{Ker}}
\DeclareMathOperator{\ann}{\mathrm{Ann}}
\DeclareMathOperator{\Ext}{\mathrm{Ext}}

\DeclareMathOperator{\Tor}{{\mathscr{T}}\!{\textit{or}}}

\newcommand{\per}{{\operatorname{per}}}
\newcommand{\hdot}{{\:\raisebox{2pt}{\text{\circle*{1.5}}}}}

\renewcommand{\mod}{{\,\operatorname{\textsl{mod}}\ }}
\newcommand{\dr}{{\operatorname{DR}}}
\newcommand{\erem}{\hfill$\lozenge$\end{rem}}
\newcommand{\eerem}{\hfill$\lozenge$\end{rem}\vskip 3pt }
\newcommand{\dis}{\displaystyle}

\newcommand{\beq}{\begin{equation}\label}
\newcommand{\eeq}{\end{equation}}
\newcommand{\hb}{{\hbar}}
\newcommand{\ohb}{{\oo_\hbar}}
\newcommand{\wh}{\widehat }
\newcommand{\wt}{\widetilde }
\newcommand{\fra}{\mathfrak }
\newcommand{\cal}{\mathcal }
\newcommand{\A}{{\ca A}}
\newcommand{\GG}{{\ca G}}
\newcommand{\Om}{\Omega }
\newcommand{\oo}{{\mathcal O}}

\renewcommand{\k}{\Bbbk }
\newcommand{\kt}{{\Bbbk^\times}}
\newcommand{\iso}{{\;\stackrel{_\sim}{\to}\;}}

\def\ccirc{{{}_{\,{}^{^\circ}}}}
\newcommand{\bo}{\mbox{$\bigotimes$}}
\renewcommand{\o}{\otimes }

\newcommand{\en}{{\enspace}}

\newcommand{\vi}{${\en\sf {(i)}}\;$}
\newcommand{\vii}{${\;\sf {(ii)}}\;$}

\newcommand{\half}{\mbox{$\frac{1}{2}$}}
\newcommand{\om}{\omega }

\newcommand{\la}{\langle }
\newcommand{\ra}{\rangle }
\newcommand{\al}{\alpha }
\renewcommand{\a}{{\mathfrak a}}
\newcommand{\g}{{\mathfrak g}}
\newcommand{\hc}{ Harish-Chandra }
\newcommand{\dd}{{{}^{^{_\D\!}\!}\delta}}
\newcommand{\um}{{}^{^{_\mm\!}\!}}
\newcommand{\ud}{{}^{^{_\D\!}\!}}
\newcommand{\md}{{{}^{^{_\mm\!}\!}\delta}}
\newcommand{\pp}{{\mathcal P}}
\newcommand{\forget}{{\mathsf{F}}}
\newcommand{\sset}{\subset}
\newcommand{\sminus}{\smallsetminus}
\newcommand{\intoo}{\,\xymatrix{\ar@{^{(}->}[r]&}\,}
\newcommand{\ontoo}{\,\xymatrix{\ar@{->>}[r]&}\,}
\newcommand{\into}{\,\hookrightarrow\,}
\newcommand{\too}{\,\longrightarrow\,}
\newcommand{\mto}{\mapsto}
\newcommand{\onto}{\,\,\twoheadrightarrow\,\,}
\newcommand{\eps}{\varepsilon }
\renewcommand{\v}{{\mathfrak v}}
\newcommand{\y}{{\mathfrak x}}

\newcommand{\fp}{{\mathfrak p}}
\newcommand{\fu}{{\mathfrak u}}
\newcommand{\gl}{{\mathfrak{gl}}}
\newcommand{\pa}{\partial }
\newcommand{\Oom}{{\Omega^{\geq 1}_Y}}
\newcommand{\pr}{{\mathrm{pr}}}
\newcommand{\io}{{\iota}}

\newcommand{\D}{\mathcal{D}}

\newcommand{\G}{\mathfrak{G}}
\newcommand{\T}{\mathcal T}
\newcommand{\M}{\mathcal M}
\newcommand{\J}{\mathcal{J}}
\newcommand{\mm}{\mathcal M}
\newcommand{\ca}{\mathcal}

\newcommand{\oa}{{\overline{\Aut}(\D,\mm)}}
\newcommand{\od}{{\overline{\Der}(\D,\mm)}}

\newcommand{\li}{{\geq i}}

\renewcommand{\sp}{{\mathfrak{sp}}}

\newcommand{\oh}{\mbox{$\frac{1}{\hb}$}}

\DeclareMathOperator{\Loc}{\mathrm{Loc}}

\DeclareMathOperator{\Der}{\mathrm{Der}}
\DeclareMathOperator{\Aut}{\mathrm{Aut}}

\DeclareMathOperator{\Tr}{\mathrm{Tr}}
\DeclareMathOperator{\Lie}{\mathrm{Lie}}

\begin{document}

\title{Quantization of Line Bundles on\\
 Lagrangian Subvarieties}
\author{Vladimir Baranovsky}
\address{V.B.:
Department of Mathematics, University of California at Irvine, 340 Rowland 
Hall, Irvine, CA 92617, USA.}
\email{vbaranov@math.uci.edu}
\author{Victor Ginzburg}
\address{V.G.:
Department of Mathematics, University of Chicago,  Chicago, IL 
60637, ~USA.}
\email{ginzburg@math.uchicago.edu}
\author{Dmitry  Kaledin}
\address{D.K.: 
Steklov Mathematical Institute, Algebraic Geometry Section,
Gubkina, ~8, Moscow, 119991, Russia.}
\email{kaledin@mi.ras.ru}
\author{Jeremy Pecharich}
\address{J.P.: 
Department of Mathematics, Pomona College, 640 North College Avenue, Claremont, CA 91711, USA}
\email{jpechari@gmail.com}
\dedicatory{To Volodya Drinfeld on the occasion of his 60th Birthday}
\begin{abstract}
We apply the technique of formal geometry to give a necessary and
sufficient condition for a line bundle supported on a smooth Lagrangian
subvariety to deform to 
a sheaf of modules over a fixed deformation quantization of the structure sheaf of an algebraic symplectic variety. 
\end{abstract}
\maketitle

{\small
\tableofcontents
}
\section{Introduction}
\subsection{}
Let $X$ be a smooth algebraic symplectic variety over a field $\k$ of
characteristic zero. Let $\om$ denote the  symplectic 2-form 
and $\{-,-\}$ the associated Poisson bracket on  $\oo_X$, the
 structure sheaf of $X$. 
A formal quantization of $X$ is, by definition, a
 sheaf $\oo_\hb$ on $X$ (in the Zariski topology)
 of  flat associative $\k[[\hb]]$-algebras (which is complete and separated in the linear
topology of a  $\k[[\hb]]$-module), 
equipped with an isomorphism $\oo_\hb/\hb\oo_\hb\cong \oo_X$
and such that $\frac{1}{\hb}(ab-ba)\mmod\hb=\{a\mmod\hb,\, b\mmod\hb\}$ for
all $a,b\in\oo_\hb$.

We will be interested in the problem
of  quantization of $\oo_X$-modules. 
Thus,  given  a coherent $\oo_X$-module $\cal L$ and a formal quantization
$\ohb$ of $X$, we are looking for $\ohb$-modules $\cal L_\hb$, flat over 
$\k[[\hb]]$, 
such that $\cal L_\hb/\hb\cal L_\hb\cong \cal L$.
A necessary condition for the existence of such an $\cal L_\hb$
is provided by
the fundamental {\em integrability of characteristics} theorem,
due to Gabber \cite{Ga}. It says that if $\cal L$ admits a flat deformation
to an $\oo_\hb/\hb^3\oo_\hb$-module then the (smooth locus of) every irreducible component of
$\supp\cal L$, the support of the coherent sheaf $\cal L$,
must be a coisotropic subvariety of $X$. 

In this paper
  we restrict ourselves to a special case where the support of
$\cal L$ is a  smooth Lagrangian subvariety $Y\sset X$ and, moreover,
$\cal L$  is a direct image to $X$
of  (the sheaf of sections of) a line bundle $L$ on $Y$.
Our main result provides a complete classification
of  all formal quantizations $\cal L_\hb$, of $\cal L$,  in terms of the  
Atiyah-Chern class
of $L$ and a `noncommutative period map'
introduced by Bezrukavnikov and Kaledin ~\cite{BK}.

In more detail,   let $Y$ be  a  smooth Lagrangian subvariety of $X$,
let $\Om^{\geq 1}_Y$ denote
the truncated de Rham complex
$0\to \Om^1_Y\to\Om^2_Y\to\ldots$,
 a subcomplex of the algebraic de Rham complex
$(\Om^\hdot_Y,d)$. Thus, one has a canonical
map $H^2(\Oom)=H^2(Y,\Oom)\to H^2(Y,\Om^\hdot_Y)=H^2_\dr(Y),\ \al\mto \al_\dr$.
For any line bundle $L$ on $Y$ there is 
an associated Atiyah class $c_1(L)\in H^2(\Oom)$
such that its image $c_1(L)_\dr$ is the usual first Chern
class of $L$.

In section \ref{J-sec} we construct a class $At(\oo_\hb,Y)\in H^2(Y,\Oom)$ canonically associated
with any quantization $\ohb$ of $\oo_X$.
This class corresponds to a natural Atiyah algebra
\beq{pic_J}0\to\oo_Y\to
\Tor^{\oo_\hb}_1(\oo_Y,\oo_Y)\to\T_Y\to0,
\eeq
where $\oo_Y$ is viewed as an $\oo_\hb$-module
via the projection $\oo_\hb\to\oo_X$.


On the other hand, Bezrukavnikov and Kaledin have constructed, see
 \cite[Definition 4.1]{BK}, a class
\beq{per}\operatorname{per}(\ohb)=[\om] + \hb\omega_1(\ohb) + \hb^2 \omega_2(\ohb) + \hb^3 \omega_3(\ohb) +
\ldots\  \in H^2_\dr(X)[[\hb]],
\eeq
where $[\om]$ stands for the de Rham cohomology class
of the closed 2-form $\om$ and $\om_i(\ohb)\in H^2_\dr(X)$. We show 
in Lemma \ref{alat} that the two constructions
are compatible in the sense that one has
\beq{compat}
At(\ohb,Y)_\dr=i^*_Y(\om_1(\ohb)),
\eeq
where $i^*_Y:\ H^2_\dr(X)\to H^2_\dr(Y)$ is the restriction
map induced by the imbedding $i_Y: Y\into X$.

Now, let $K_Y=\Om^{\dim Y}_Y$ denote the canonical bundle.
Our main result reads

\begin{theorem}
\label{maintheorem}
Let  $(X,\om)$ be an algebraic symplectic manifold, let
$i_Y: Y\into X$ be  a  closed imbedding of a smooth Lagrangian
subvariety, and
let $L$ be the sheaf of sections of a line bundle on $Y$. Then, we have

\vi The sheaf $(i_Y)_* L$  admits a quantization, i.e. there exists
a complete flat left $\ohb$-module $\cal L_\hb$ such that
$\cal L_\hb/\hb \cal L_\hb\cong (i_Y)_* L$  if and only if the following 
two conditions hold:
\begin{align*}
&c_1(L) - \frac{1}{2} c_1(K_Y) =
At(\ohb,Y)\quad\text{holds in}\en H^2(\Om^{\geq1}_Y);\\
&i_Y^* \omega_i(\ohb)=0\en\text{holds in}\en  H^2_\dr(Y), \quad
\forall\en i\geq 2.
\end{align*}

\vii 
If the set $Q(X,\om, Y)$ of isomorphism classes of  quantizations 
of line bundles on $Y$ is non-empty, then this set
affords a free and
transitive action of the group of isomorphism classes of $(\mathcal{O}_Y[[\hb]])^*$-torsors on $Y$
with a flat algebraic connection. 
\end{theorem}

For a line bundle $L$ the Atiyah-Chern class $c_1(L)\in H^2(\Om^{\geq1}_Y)$
 measures
the obstruction to the existence of a flat algebraic connection on $L$. Therefore, 
in a special case where $At(\ohb,Y) =0$ and $i_Y^* \omega_i(\ohb)=0, 
i \geq 2$,
our theorem claims that $L$ admits a deformation 
quantization if and only if the line bundle $L^{\otimes 2} \otimes K_Y^\vee$ has a
flat algebraic connection. In such a case,
one says, abusing terminology somewhat,  that
$L$ is a square root of the canonical class. 

The existence of  quantization for  square roots of
the canonical bundle seems to have been first discovered  (without proof)
by M. Kashiwara \cite{Ka}, in the framework
of complex analytic contact  geometry.
Later on,  D'Agnolo and Schapira \cite{DS} established
a similar result for
Lagrangian submanifolds of  a complex analytic
symplectic manifold. In the $C^\infty$-context, 
some closely related constructions can be found
in the work of Nest and Tsygan \cite{NT}.

\begin{rem}
Our main result can be applied in a slightly more general setting where
$Y$ is 
a (possibly singular)
{\em normal} subvariety of $X$
such that $Y^{\text{reg}}$, the smooth locus of $Y$,
is a Lagrangian submanifold,
and where   $L$ is  a coherent $\oo_Y$-module
 isomorphic to 
the direct image of a line bundle $L^{\text{reg}}$ 
on  $Y^{\text{reg}}$. In this case, our theorem 
tells when $L^{\text{reg}}$ admits a deformation quantization
$\cal L^{\text{reg}}_\hb$. This is a sheaf on $X\sminus(Y\sminus
Y^{\text{reg}})$
and the direct image of $\cal L^{\text{reg}}_\hb$
to $X$ is a coherent $\oo_\hb$-module,
since $\dim(Y\sminus Y^{\text{reg}})\leq \dim Y-2$.
It is clear  that the latter module provides  a deformation
quantization of 
the original sheaf $L$. 
\end{rem}

\subsection{Acknowledgements}
We are grateful to Dima Arinkin for helpful remarks
and to Pierre Schapira for  historical comments.

The first author was supported by a Simons Foundation Collaboration 
Grant. 
The second author  was supported in part by the NSF grant DMS-1001677.
The third author was partially supported by the RFBR grant
~{12-01-33024}, Russian Federation government grant,
ag. 11.G34.31.0023, and the Dynasty Foundation Award. 
The fourth author would  like to thank K. Behrend and B. Fantechi
for many inspiring discussions on quantization. He is grateful to
 the Mathematical Science Research
Institute, Berkeley, for hospitality
during his stay in the Spring of 2013.
The work on the paper was also partially supported  by the NSF Grant
DMS-0932078 000.

\section{Basic constructions}
\label{sec2}

\subsection{}\label{setup}
Fix $n\geq 1$ and a $2n$-dimensional vector space $\v$
equipped with a symplectic form  $\om\in \wedge^2(\v^*)$.

Associated with $\om$, there is a Heisenberg Lie algebra
with an underlying vector space
$\v\oplus\k\hb$, where $\hb$ denotes
a fixed base element, and the Lie bracket \red{is} defined by the
formulas 
\[[x,y]=\om(x,y)\cdot\hb,\quad [x,\hb]=0,\quad\forall x,y\in \v.\]

Let $D$ be the universal enveloping algebra of this  Heisenberg Lie
algebra. Thus, $D$ is  an associative $\k[\hb]$-algebra,
also known as the `homogeneous version of the Weyl 
algebra'. 
The algebra $D$ comes equipped with a natural grading 
$D=\oplus_{i\geq 0}\ D^i$ such that the vector space
$\v$ is placed in degree 1
and the element $\hb$ is assigned grade degree 2. 
Let $\D=\prod_{i\geq 0}\ D^i$ and, for any 
$j\geq 0$, put $\D^{\geq j}=\prod_{i\geq j}\ D^i$.
Thus, $\D$ is an associative $\k[[\hb]]$-algebra
equipped with a descending filtration
$\D=\D^{\geq 0}\supset \D^{\geq 1}\supset\ldots$,
by two-sided ideals. This filtration
makes $\D$ a
 complete topological algebra with $\D^{\geq 1}$ being
the unique maximal ideal of $\D$.

Let $\oh\D$ denote a free rank one $\D$-submodule of $\k((\hb))\o_{\k[[\hb]]}\D$
generated by the element $\oh$.
Let $\G^i:=\oh \D^{\geq i+2}$. 
The commutation relations in $\D$ imply that
$[\G^i,\G^j]\sset \G^{i+j}$, where
$[a,b]=ab-ba$ denotes the commutator in the algebra
$\k((\hb))\o_{\k[[\hb]]}\D$, so that $\G$ is a graded Lie algebra.
We have
$\G^{-2}=\k$ and  $\G^{-1}=\v$.
The vector space $\G^{-1}\oplus \G^{-2} \subset \G$ is a Lie subalgebra
isomorphic to the Heisenberg algebra.

The homogeneous component $\G^0 \subset \G$ is also
a Lie subalgebra. There is
a canonical direct sum decomposition
$\G^0=\k_\G\oplus[\G^0,\G^0]$, where
the first summand, the image of the 
imbedding $\eps_\G:\k\into \G^0,\ c\mto \oh(c\hb)$,
is the center of the Lie algebra $\G^0$.
The commutator map $[-,-]:\ \G^0\times \G^{-1}\to\G^{-1}$
gives an action of the Lie algebra $\G^0$ on
$\G^{-1}=\v$. The center $\k_\G\sset\G^0$ acts trivially;
the action of the second summand yields
a canonical Lie algebra isomorphism
\beq{sigma}
\sigma:\ \sp(\v)\ \iso\ [\G^0,\G^0]\sset \oh \D^2,
\eeq
where $\sp(\v)$ is the Lie algebra of the
symplectic group $Sp(\v)$.

We put $\G=\oh\D/\oh\k$, a quotient
of the Lie algebra $\oh\D$ by a central
subalgebra. The filtration on $\D$ induces 
 a descending filtration
$\G=\G^{\geq -1}\supset\G^{\geq 0}\supset\G^{\geq 1}\supset\ldots$,
where  $\G^\li=\prod_{j\geq i}\ \G^j$.
Here, $\G^{\geq0}$ is a Lie subalgebra of $\G$
and $\G^{\geq1}$ is a pronilpotent
Lie ideal of  $\G^{\geq0}$, moreover,  the natural
map $\G^0\to\G$ provides  a canonical
`Levi decomposition': 
\beq{levi}
\G^{\geq0}=\k_\G\ \oplus\ \big(\sp(\v)\ltimes\G^{\geq1}\big).
\eeq

\subsection{}
We recall the following definition, see \cite{BB}.

\begin{definition}
A \emph{Harish-Chandra pair} over $\k$ is a pair $\langle G, \frak
h\rangle$ \red{where} \blue{with} $G$
\red{is}  a connected affine (pro)algebraic group, $\frak h$ \red{is} a Lie
algebra equipped with a $G$-action, and \red{with a Lie algebra}
\blue{an} embedding $\frak
g\to \frak h$ of the Lie algebra $\frak g$ of $G$ such that the adjoint
action of $\frak g$ on $\frak h$ \red{via the imbedding equals}
\blue{is} the differential of the given
$G$-action.  
\end{definition}

The reason for introducing the notion of a Harish-Chandra pair is that
there exist infinite dimensional Lie algebras which cannot be
integrated to algebraic groups. However, they have Lie subalgebras which
can be integrated to an algebraic group. 
\red{Here is an important example.}

Let $\Der(\D)$ denote the Lie algebra of
$\k[[\hb]]$-linear continuous derivations of
the ring $\D$ and  $\Der^{\geq j}(\D)\sset \Der(\D)$  the space
of derivations $\delta: \D\to\D$
such that $\delta(\D^{\geq1})\sset \D^{\geq j+1}$.
It is clear that $\Der^{\geq 0}(\D)$ is a Lie
subalgebra and $\Der^{\geq1}(\D)$ is a pronilpotent Lie ideal
of that subalgebra.

Let $\Aut(\D)$ be the group of $\k[[\hb]]$-linear automorphisms
of the algebra $\D$. The group $\Aut(\D)$ has the natural
structure of a proalgebraic group with Lie algebra
$\Lie\Aut(\D)=\Der^{\geq0}(\D)$. 
The pair $\la\Aut(\D),\ \Der(\D)\ra$ is  a Harish-Chandra pair.

The group $Sp(\v)$ acts on the Heisenberg algebra
$\v\oplus\k\hb$,
hence also on the objects $\D,\G,\G^{\geq1}$, and $\GG^{\geq1}$, by
automorphisms.
In particular, one has a natural group
homomorphism $\Theta_\D: Sp(\v)\to\Aut(\D)$
and an associated  Lie algebra homomorphism 
$\theta_\D: \sp(\v)\to \Der(\D)$.
It is well-known, and easy to verify, that 
the map $\theta_\D$ is related to the map in \eqref{sigma}
by the equation
\beq{sisi}
\theta_\D(a)(x)=[\sigma(a), x],
\qquad\forall a\in\sp(v), x\in\oh\D.
\eeq

Next, let $\GG^{\geq1}$ be a  prounipotent group
associated with $\G^{\geq1}$, a pronilpotent Lie
algebra. We have the natural central imbedding
$\hb\k[[\hb]] \into \G^{\geq 1}$,
of pronilpotent Lie algebras. This induces
an injective  homomorphism $1+\hb\k[[\hb]] \into \GG^{\geq 1}$,
of the corresponding prounipotent groups.

Folowing \cite{BK}, we consider the group
$\GG:=\k^\times_\GG\ \times\ \big(Sp(\v)\ltimes\GG^{\geq1}\big)$.
Here, the semidirect product $Sp(\v)\ltimes\GG^{\geq1}$ is taken with respect to
the natural $Sp(\v)$-action on  $\G^{\geq1}$
by automorphisms. The group $\GG$ has the structure of a proalgebraic group.

A  cartesian product of the natural maps $\k\to\k_\GG$
and $1+\hb\k[[\hb]] \to \GG^{\geq 1}$ provides, via the
obvious
isomorphism $\k[[\hb]]^\times =\k^\times \times (1+\hb\k[[\hb]])$, 
a central imbedding $\k[[\hb]]^\times\into \GG$,
of proalgebraic groups.
By construction, one also has a canonical 
imbedding
$\Sigma: Sp(\v)\into\GG$.
One may view the group $\k^\times_\GG\times \Sigma(Sp(\v))$ as
a Levi subgroup of $\GG$.
By \eqref{levi},
we have an isomorphism
$\Lie\GG = \k \oplus \big(Sp(\v)\oplus\G^{\geq1}\big)\cong \G^{\geq
  0}$.
Furthermore, it follows from equation \eqref{sisi} that the natural imbedding
$\Lie\GG\cong \G^{\geq 0}\into\G$ makes the pair
$\la \GG,\ \G\ra$ a  Harish-Chandra pair.

The space $\D$ is a Lie ideal in $\oh\D$.
Hence, there is a well-defined 
action $\oh\D\times\D\to\D,\ a\times x\mto [a,x]$.
This action descends to a Lie algebra
homomorphism $\varphi_\D:\ \G\to \Der(\D)$
with kernel  $\oh\k[[\hb]]/\oh\k=\k[[\hb]]$.
It is easy to see that for all $i\geq0$, one has
$\varphi_\D(\G^{\geq i})\sset\Der^{\geq i}(\D)$.
For $i=1$, exponentiating the map $\G^{\geq 1}\to \Der^{\geq
  1}(\D)$,
of pronilpotent Lie algebras, yields a homomorphism
$\Phi_\D^1: \GG^{\geq 1}\to \Aut^{\geq 1}(\D)$,
of the corresponding prounipotent groups. 
One can further extend the latter homomorphism
to a homomorphism 
\[\Phi_\D:\ \GG=\k^\times_\GG\, \times\,
\big(Sp(\v)\ltimes\GG^{\geq1}\big)
\ \to\ \Aut(\D),\quad c\times (g\ltimes u)\mto \Theta(g)\Phi_\D^1(u).\]
The differential of $\Phi_\D$ agrees with the  homomorphism
$\varphi_\D: \G^{\geq 0}\to \Der^{\geq 0}(\D)$, of Lie algebras.
Moreover,
as has been observed in \cite{BK}, the  maps $\Phi_\D$ and $\varphi_\D$
give rise to a central
 extension
\beq{aut}
1 \to \la\k[[\hb]]^\times,\  \k[[\hb]]\ra\too \la\GG,\
\G\ra \stackrel{\la\Phi_\D,\varphi_\D\ra}\too
\la\Aut(\D),\ \Der(\D)\ra \to 1
\eeq
of Harish-Chandra pairs, see \cite[Section 3.3]{BK}.

\subsection{} 

It is often convenient to choose  a basis 
$x_1,\ldots, x_n,y_1,\ldots, y_n$, of $\v$, such that 
\beq{xy}\om(x_i,x_j)=0=\om(y_i,y_j),\qquad \om(x_i,y_j)=\delta_{ij},\quad
\forall i,j=1,\ldots,n.
\eeq 

The algebra $\D$ is (topologically) generated by $\hb$ and the 
basis elements of $\v$
 subject to the commutation
relations 
\[
[x_i,x_j]=0=[y_i,y_j],\quad [y_j, x_i] = \delta_{ij} \hb,\quad [\hb,
x_i] = [\hb, y_j]=0,\en\forall i,j.
\]
In particular, one has a canonical algebra imbedding
$\k[[x_1,\ldots, x_n,\hb]]\into \D$,
resp.
$\k[[y_1,\ldots, y_n,\hb]]\into \D$. 
Further,
one proves  the following  identity in $\D$: 
\beq{commutation}
y_i\cdot f- f\cdot y_i=\hb\cdot\pa_i f,\qquad
\forall f\in  \k[[x_1,\ldots,x_n]],\en i=1,\ldots,n,
\eeq
where we have used the notation $\pa_i=\frac{\pa}{\pa x_i}$.

Let $\A=\D/\hb\D$. This is a complete topological commutative algebra that comes
equipped with a Poisson bracket defined
by the formula 
\[\{a\mmod\hb,\ b\mmod\hb\}=(\oh [a,b])\mmod\hb,\qquad\forall
a,b\in\D.\]
There is a natural isomorphism
$\ca A\cong\k[[x_1,\ldots,
x_n,y_1,\ldots, y_n]]$, of topological algebras.
The resulting Poisson bracket on $\k[[x_1,\ldots,
x_n,y_1,\ldots, y_n]]$ has the standard form
\beq{bracket}
\{f,g\}\ =\ 
\sum_{1\leq i\leq n}\ \left(\frac{\pa f}{\pa x_i}\frac{\pa g}{\pa y_i}-
\frac{\pa f}{\pa y_i}\frac{\pa g}{\pa x_i}\right).
\eeq
The algebra $\D$ may be viewed as a quantization
of the Poisson algebra $\A$.

From now on, we fix a Lagrangian subspace $\y\sset\v$.
Let $\mm=\D/\D\y$. This is
 a  left $\D$-module and we have
$\mm/\hb\mm\cong\A/\A\y$ as
an $\A$-module. Let $1_\mm=1\mmod\D\y$ denote the
generator of $\mm$.

We will assume (as we may) that
the symplectic basis of $\v$, cf. \eqref{xy}, is chosen in such a way that
$y_1,\ldots,y_n$ is a basis of $\y$.
It follows readily that the composite
$\k[[x_1, \ldots, x_n, \hb]]\into\D\onto\D/\D\y=\mm$
is an isomorphism of $\k[[x_1, \ldots, x_n, \hb]]$-modules.
Using \eqref{commutation} one finds that the
action of $y_j$ on $\mm$ goes, via the isomorphism,
to the operator  $\hb \pa_i$ on $\k[[x_1, \ldots, x_n, \hb]]$.

The group $P:=\{g\in Sp(\v)\mid g(\y)\sset \y\}$
 is a parabolic subgroup of $Sp(\v)$.
Let
$\fp=\Lie P\sset\sp(\v)$
be the corresponding parabolic subalgebra,
 $\fu$ the nilradical of $\fp$.
Restriction to the subspace $\y\sset\v$ gives
a  map  $\fp\to\gl(\y),\ a\mto a|_\y$, that induces
a  canonical isomorphism 
$\fp/\fu\cong \gl(\y)$, of Lie algebras.

By \eqref{sigma}, we have the map $\sp(\v)\to \D^2,\ a\mto \hb\sigma(a)$.
The following  formula is well-known.

\begin{lemma}\label{1m} For any $a\in\fp$, we have
$\ (\hb\,\sigma(a))(1_\mm)=\frac{1}{2}\Tr(a|_\y)\cdot 1_\mm$.
\end{lemma}
\begin{proof} Using the basis $x_1, \ldots, x_n,y_1,\ldots,y_n$,
the Lie algebra $\fp$ may be described as a subalgebra of $\sp(\v)$
formed by the matrices
$$
a=\left(\begin{array}{cc}
g& h\\
0&-g^T
\end{array} \right),\qquad g=\|g_{ij}\|,\ h=\|h_{ij}\|\en\text{where }
h_{ij}=h_{ji}, \
\forall i,j=1,\ldots,n.
$$ 
For such a matrix $a$, one finds:
\[\hb\cdot\sigma(a)=
\half\sum_{ij}g_{ij}(x_iy_j+y_jx_i)+\half\sum_{ij} h_{ij} y_iy_j.
\]

The formula of the lemma follows from this by a straightforward
computation using that $y_i(1_\mm)=0$.
\end{proof}

\begin{lemma}\label{wellknown}
Let $M$ be a complete topological finitely generated left
$\D$-module without $\hb$-torsion.
Then, any isomorphism $M/\hb M\cong\A/\A\y$,
of $\A$-modules, can be lifted to an isomorphism
 $M\cong \mm$ of $\D$-modules.
\end{lemma}
\begin{proof} 
We will use the identification $\mm\cong\k[[x_1, \ldots, x_n, \hb]]$,
resp. $\A/\A\y\cong\k[[x_1, \ldots, x_n]]$.
Let $1_M\in M$ be any element that maps to
$1\in  \k[[x_1, \ldots, x_n]]$ under the
composition $M\onto M/\hb M\iso  \k[[x_1, \ldots, x_n]]$.
 Then, the map $p: \k[[x_1, \ldots, x_n,
\hb]]\to M,\  u\mto u(1_M)$ 
is a $\k[[x_1, \ldots, x_n,\hb]]$-linear map that induces
a bijection $\k[[x_1, \ldots, x_n]]\iso M/\hb M$.
It follows by Nakayama's lemma that the map $p$
is  an isomorphism of $\k[[x_1, \ldots, x_n,\hb]]$-modules.
Therefore, for each $j=1,\ldots,n$, there exists
a unique element $f_j\in \k[[x_1, \ldots, x_n,\hb]]$
such that $y_j(1_M)=f_j(1_M)$. Furthermore,
the power series $f_j$ is divisible by $\hb$,  since
$y_j(1_M)\mmod \hb M=0$.
Thus,
we have $f_j=\hb\cdot g_j$ for some $g_j\in\k[[x_1, \ldots, x_n,\hb]]$.

For any $1\leq i,j\leq n$, using
the commutation relation
in \eqref{commutation},  we find
$$ y_iy_j(1_M)=y_i(f_j(1_M))=
\hb\pa_i f_j(1_M)+f_j(y_i(1_M))=
\hb^2\bigl(\pa_i g_j(1_M)+g_j\cdot g_i(1_M)\bigr).
$$
Since $y_iy_j=y_jy_i$, we deduce that
$\pa_ig_j=\pa_j g_i$ for all $i,j$.
Hence, there exists a formal power series
$g\in \k[[x_1, \ldots, x_n,\hb]]$ such that we have
$g_j=\pa_j g$ for all $j$.
Furthermore,
we may (and will) choose $g$ to be contained in
the ideal generated by the elements
$x_1,\ldots,x_n$. Thus, the element $e^{-g}$ is a well-defined
element of $\k[[x_1, \ldots, x_n,\hb]]$.

We put $m:=e^{-g}(1_M)\in M$.
We compute
\begin{align*}
y_i(m)&=y_i(e^{-g}(1_M))=\hb\pa_i(e^{-g})(1_M)+e^{-g}(y_i(1_M))\\
&=-\hb \pa_i g\cdot e^{-g}(1_M)+ e^{-g}\cdot \hb g_i(1_M)\\
&=-\hb g_i\cdot e^{-g}(1_M)+ e^{-g}\cdot \hb g_i(1_M)=0.
\end{align*}

Hence the map $\D\to M,\ u \mto u(m)$ 
descends to a map $F: \D/\D \y\to M$.
The latter map is an isomorphism of rank 1 free $\k[[x_1, \ldots, x_n,\hb]]$-modules,
since $e^{-g}\in \k[[x_1, \ldots, x_n,\hb]]$ is an invertible element.
We conclude that $F$ is  an isomorphism of $\D$-modules
that lifts the  isomorphism $\A/\A\y\cong M/\hb M$.
\end{proof}

\section{Comparison of   Harish-Chandra pairs}
\subsection{} 

We introduce various  Harish-Chandra pairs
canonically associated to the Lagrangian subspace $\y\sset\v$.
We will use the notation $K=\k[[\hb]]$.

Let $\J\sset\D$ be the preimage of
the ideal $\A\y\sset\A$ under the natural algebra
projection $\D\onto\D/\hb\D=\A$. Thus $\J$
is a two-sided ideal of $\D$.
Let $\Aut(\D)_\J$, resp. $\Der(\D)_\J$,
be the subset of  $\Aut(\D)$, resp. $\Der(\D)$,
formed by the maps $f:\D\to\D$ such that $f(\J)\sset\J$.
The pair $\la\Aut(\D)_\J,\ \Der(\D)_\J\ra$ is a
 Harish-Chandra subpair of 
$\la\Aut(\D),\ \Der(\D)\ra$.

Let $\la\GG_\J,\ \G_\J\ra$
be the preimage of the pair $\la\Aut(\D)_\J,\ \Der(\D)_\J\ra$
under the morphism $\la\Phi_\D,\varphi_\D\ra$ in \eqref{aut}.
Thus, by construction, one has a central extension
\beq{J-ext}
1\to\la K^\times, K\ra\
 \too
\la\GG_\J,\G_\J\ra\ \stackrel{\la\Phi_\D,\varphi_\D\ra}\too\
\la\Aut(\D)_\J, \Der(\D)_\J\ra\to1
\eeq

Below, we will identify the space
$\oh\J\sset\oh\D$ with its image in $\G$ under
the composite map $\oh\J\into\oh\D\onto\oh\D/\oh D^0=\G$,
which  is
injective since $\J\cap D^0=0$. 

\begin{lemma}\label{jlem} We have $\G_\J=\oh\J$.
\end{lemma}

\begin{proof} 
Let  $a\in\G$ and put  $f:=(\hb a)\mmod \hb\D$.
We view $f$ as an element of $\A$ without
constant term.  Then, since $\J/\hb\D=\A\y$,
an inclusion $[a,\J]\sset\J$ is equivalent
to  $\{f,\A\y\}\sset\A\y$.
Using formula \eqref{bracket}, one shows
that the latter inclusion holds if and only if
one has   $f=c+f'$ for some constant $c\in\k$ and
some $f'\in \A\y$. We must have $c=0$.
Hence,   $a\in\oh\J$.
\end{proof}

Let  $\Der(\D, \mm)$ be the Lie algebra of derivations, 
resp. $\Aut(\D, \mm)$ the group of automorphisms,
of the pair $(\D,\mm)$.
By definition, an element
of   $\Der(\D, \mm)$ is a pair  $(\dd ,\md )$
where $\dd 
\in\Der(\D)$ and
$\md : \mm \to \mm$
is a continuous $\k[[\hb]]$-linear map  such that
\beq{dm}
\md (um) = \dd (u) m + u(\md (m)),
\qquad\forall u\in\D, m\in\mm.
\eeq
Similarly,  an element
of   $\Aut(\D, \mm)$ is a pair
$(\ud f, \um f)$
where $\ud f\in\Aut(\D)$ and $\um f: M\to M$ is
a continuous $\k[[\hb]]$-linear bijection  such that
$
\um f(um) = \ud f(u)\, \um f(m).
$
Thus, $\la\Aut(\D, \mm),\ \Der(\D, \mm)\ra$
is  a Harish-Chandra pair.

The
assignment $c\mto (\id_\D, c\cdot\id_\mm)$ gives
a natural central imbedding
$\eps_{\Aut}: K^\times\into \Aut(\D,\mm)$.
Similarly, one has a central Lie algebra imbedding
$\eps_{\Der}: K\into \Der(\D,\mm)$ given by
$a\mto (0, a\cdot\id_\mm)$.
This gives an injective morphism
$\eps=\la\eps_{\Aut},\eps_{\Der}\ra:\
\la K^\times, K\ra\into \la\Aut(\D,\mm),\Der(\D,\mm)\ra$
of \hc pairs.

Further,
it is clear that forgetting the action on $\mm$
yields a morphism of Harish-Chandra pairs:
\[\forget=\la\forget_{\Aut},\forget_{\Der}\ra:\
\la\Aut(\D, \mm),\ \Der(\D, \mm)\ra
\to \la \Aut(\D),\Der(\D)\ra.
\]
\begin{lemma}
One has an inclusion
\beq{inc}
\forget\bigl(\la\Aut(\D, \mm),
\Der(\D, \mm)\ra\bigr)\ \sset \
\la\Aut(\D)_\J, \Der(\D)_\J\ra.
\eeq
\end{lemma}
\begin{proof}
Observe  that  the ideal $\J$ is equal
to $\ann(\mm/\hb\mm)$, the  annihilator of the $\D$-module
$\mm/\hb\mm$, since $\mm/\hb\mm=\A/\A\y$. 
Further, it follows  from
equation \eqref{dm} that, for any pair $(\dd , \md )
\in \Der(\D, \mm)$, 
the map $\md $
takes the  annihilator of $\mm/\hb\mm$ to itself,
that is, takes $\J$ to $\J$. We deduce that the image of
the map $\forget_{\Der}: \Der(\D, \mm) \to \Der(\D)$
 is contained in 
$\Der(\D)_{\ca J}$. A similar argument
yields the  inclusion involving the map
$\forget_{\Aut}$, proving \eqref{inc}.
\end{proof}

Next we note that since $\J=\ann(\mm/\hb\mm)$,
one has $\J\mm\sset \hb\mm$. Hence, there is a well-defined 
action $\oh\J\times\mm\to \mm,\ \oh u\times m\mto \oh um$,
of the Lie algebra $\oh\J$ on $\mm$.
For $a\in \oh\J$, let $\varphi_\mm(a): \mm\to\mm$
denote the map $m\mto am$. 
It is immediate to check
that equation \eqref{dm} holds for
 the pair of maps $\la\varphi_\D(a),\varphi_\mm(a)\ra$,
i.e., this pair
gives an element of  $\Der(\D, \mm)$.
We deduce that the assignment $a\mto \la\varphi_\D(a),\varphi_\mm(a)\ra$
yields a  map 
$\varphi_{\D,\mm}: \oh\J\to \Der(\D, \mm)$. 
 Note that $\varphi_{\D,\mm}|_K=\eps_{\Der}$.

\begin{lemma}\label{delta_DM}
The map $\varphi_{\D,\mm}$ is a Lie algebra isomorphism.
\end{lemma}
\begin{proof}
Let $\varphi_\J$ be the restriction of the map
$\varphi_\D$ to the subalgebra $\G_\J\sset \G$.
Using  Lemma \ref{jlem}
 we obtain the following diagram

\beq{mod-ext}
\xymatrix{
K\ar@{^{(}->}[r]&\oh\J\ar[d]_<>(0.5){\varphi_{\D,\mm}}
\ar@{=}[rr]^<>(0.5){\text{Lemma  \ref{jlem}}}&& \G_\J
\ar@{->>}[d]_<>(0.5){\varphi_\J}\\
&\Der(\D, \mm)\ar[rr]^<>(0.5){\forget_{\Der}}&&
\Der(\D)_{\ca J}
}
\eeq

It is immediate from definitions that
the square in the diagram commutes.
The map 
$\varphi_\J$ being surjective,  
it follows that  the map $\forget_{\Der}$ is surjective.
The kernel of this map 
is formed by the pairs $(0, \md )$
where the map  $\md :\mm\to\mm$  commutes with the
$\D$-action, i.e. is a $\D$-module endomorphism. All
$\D$-module endomorphisms of $\mm$ are 
given by multiplication 
by an element of $K$. 
We deduce that $\Ker\forget_{\Der}=\varphi_{\D,\mm}(K)$.
Also, since $\Ker\varphi_\J=K$ 
we  get $\Ker\varphi_{\D,\mm}\sset K$. 
Furthermore, it is clear that multiplication by a nonzero
element of $K$ gives a nonzero endomorphism of $\mm$.
We conclude that the map $\varphi_{\D,\mm}$ is injective.

Now, let $\delta=(\dd,\md)$ be an element of $\Der(\D, \mm)$.
The map $\varphi_\J$ being surjective, there exists $a\in\oh\J$
such that
we have $\dd=\varphi_\J(a)$.
We know that $\varphi_{\D,\mm}(a)\in\Der(\D, \mm)$
and it is clear that $\forget_{\Der}(\delta-\varphi_{\D,\mm}(a))=0$.
It follows that $\delta-\varphi_{\D,\mm}(a)=\varphi_{\D,\mm}(c)$ for
some $c\in K$
and, hence, we have $\delta=\varphi_{\D,\mm}(a+c)$. We deduce that the
injective map $\varphi_{\D,\mm}$ is also surjective, 
proving the lemma.
\end{proof}

As a consequence of the above proof we obtain

\begin{corollary}\label{DM-exact} The following sequence is exact
\[1\to \la K^\times,K\ra
\stackrel{\eps}\too
\la\Aut(\D, \mm),\Der(\D,\mm)\ra\stackrel{\forget}\too
\la\Aut(\D)_\J, \Der(\D)_\J\ra\to1.
\]
\end{corollary}

\subsection{} 
\label{sec3.2}
Below, we identify the
subgroup $\{\pm1\}\sset \k^\times$ with its
images in $\eps_{\Aut}(\k^\times)$ and $\k^\times_\GG$,
respectively.

The main result of this section is the
following

\begin{proposition}\label{Lie-extensions}
There is  an isomorphism 
\[
\Phi_{\D,\mm}:\ \GG_{\ca J}/\{\pm1\}\iso\Aut(\D, \mm)/\{\pm1\},
\]
of proalgebraic groups,
that fits into a commutative diagram 
$$
\xymatrix{
\big\la \mbox{$\frac{K^\times}{\{\pm1\}}$}, K\big\ra\
 \ar@{^{(}->}[r]\ar@{=}[d]^<>(0.5){\id}
&\big\la\mbox{$\frac{\GG_\J}{\{\pm1\}}$},\G_\J\ra\ar@{->>}[r]^<>(0.5){\la\Phi_\D,\varphi_\D\ra}\ar[d]^<>(0.5){\la\Phi_{\D,\mm},\varphi_{\D,\mm}\ra}_<>(0.5){\cong}
&
\big\la\Aut(\D)_\J, \Der(\D)_\J\big\ra\ar@{=}[d]^<>(0.5){\id}\\
\big\la \mbox{$\frac{K^\times}{\{\pm1\}}$}, K\big\ra\
\ar@{^{(}->}[r]^<>(0.5){\eps}& 
\big\la\mbox{$\frac{\Aut(\D, \mm)}{\{\pm1\}}$},
\Der(\D, \mm)\big\ra\ar@{->>}[r]^<>(0.5){\forget}& 
\big\la\Aut(\D)_\J, \Der(\D)_\J\big\ra 
}
$$
of central extensions of Harish-Chandra pairs.
\end{proposition}

The rest of the section is devoted to the proof of the proposition.
\smallskip

Let $\G_\J^{\geq i}=\G^{\geq i}\cap\G_\J$.
It is clear that $\G_\J^{\geq 0}$ is a Lie subalgebra of $\G_\J$,
and $\G_\J^{\geq 1}$ 
is a pronilpotent
 Lie ideal of  $\G_\J^{\geq 0}$. 
Thus there is  a  prounipotent normal subgroup $\GG^{\geq 1}_\J
\sset\GG_\J$ that corresponds
to the  ideal $\G_\J^{\geq 1}$. 

\begin{claim}\label{1semidirect} We have 
$$
\GG_\J=\k^\times_\GG \times \bigl(\Sigma(P)\ \ltimes\ \GG^{\geq 1}_\J\bigr),
$$
where $\Sigma$ was introduced in the paragraph following \eqref{sisi} and the 
group $\GG_\J$ in the beginning of this section. 
In particular, one has  $\Lie\GG_\J=\G_\J^{\geq 0}$.
\end{claim}

\begin{proof}[Proof of Claim]
Observe first that, for any $g\in \Aut(\D)$,
we have $g(\D^{\geq1})\sset \D^{\geq1}$
and $g(\D^{\geq2})\sset\D^{\geq2}$,
since $\D^{\geq2}=\hb\D^{\geq1}+(\D^{\geq1})^2$.
It follows that
$g$  induces an automorphism $\tau(g)$
of the vector space $\D^{\geq1}/\D^{\geq2}=\v$.
The map $g\mto \tau(g)$ yields  a  homomorphism
$\tau: \Aut(\D)\to Sp(\v)$. Clearly, one has $\tau(\Aut(\D)_\J)\sset P$.
We also have  the  homomorphism
$\Phi_\D: \GG\to \Aut(\D)$ such that $\Phi_\D(\k^\times_\GG)=1$
and $\Phi_\D(\GG_\J)\sset \Aut(\D)_\J$, by
definition. Thus, there is a well defined composition
$\GG_\J/\k^\times_\GG\stackrel{\Phi_\D}\too
\Aut(\D)_\J\stackrel{\tau}\too P$, to be denoted by $\tau_\J$.
Note that $\G_\J^{\geq1}\subseteq\Ker(\tau_\J)$.

It is clear that we have
 $\Lie\GG_\J\sset \G_\J^{\geq 0}$. Further\blue{more},
the Lie algebra map $\tau_\J^{\Lie}: \Lie(\GG_\J/\k^\times_\GG)\to\fp$ induced by
the group homomorphism
$\tau_\J$ is equal to the composition of natural maps
\[
\Lie(\GG_\J/\k^\times_\GG)\ \into\
\G_\J^{\geq 0}/\k_\G\ \to\ 
\G_\J^{\geq 0}/(\k_\G\oplus\G_\J^{\geq 1})\ =\
(\G^0\cap \G_\J)/\k_\G\ \cong\ 
\fp.\]
where we have used that
 $\G^0\cap \G_\J=\k_\G\oplus\sigma(\fp)$,
so  $(\G^0\cap \G_\J)/\k_\J\cong\fp$. 

Using the inclusion $\G_\J^{\geq1}\subseteq\Ker(\tau_\J)$, we deduce
\[\Lie\GG_\J^{\geq1}\ \subseteq\ \Lie(\Ker(\tau_\J))\ \subseteq\ 
\Ker(\tau_\J^{\Lie})\ =\ \G_\J^{\geq1}\ =\ \Lie\GG_\J^{\geq1}.\]
This implies an equality $\Ker(\tau_\J)=\GG_\J^{\geq1}$,
since both groups are prounipotent.
The proof is completed by observing
that the map  $\Sigma|_P$ provides a section
$P\to\GG_\J\to \GG_\J/\k^\times_\GG$ 
 of the map $\tau_\J$.
\end{proof}

We use the isomorphism   $\varphi_{\D,\mm}$ and
put $\Der^{\geq i}(\D,\mm):=\varphi_{\D,\mm}(\G_\J^{\geq i})$.
Thus,
 $\Der^{\geq 0}(\D,\mm)$ is a Lie subalgebra of $\Der(\D,\mm)$,
resp.  $\Der^{\geq 1}(\D,\mm)$ is  a pronilpotent
ideal of  $\Der^{\geq 0}(\D,\mm)$. Let
$\Aut^{\geq 1}(\D,\mm)$
 be  a prounipotent subgroup of the group $\Aut(\D,\mm)$ corresponding
to the ideal $\Der^{\geq 1}(\D,\mm)$.

Next, we observe  that, for any $g\in P$,
the left ideal $\D\y$ is stable under
the map $\Theta_\D(g):\D\to\D$.
Hence, this map
descends to a map $\Theta_\mm(g):\mm\to\mm$.
The assignment $g\mto (\Theta_\D(g), \Theta_\mm(g))$
gives an injective homomorphism $\Theta_{\D,\mm}: P\to \Aut(\D,\mm)$.

\begin{claim}\label{2semidirect} We have 
\[\Aut(\D,\mm)=
\eps_{\Aut}(\k^\times) \times \bigl(\Theta_{\D,\mm}(P)\ \ltimes\ \Aut^{\geq 1}(\D,\mm)\bigr),
\]
in particular, one has  $\Lie\Aut(\D,\mm)=\Der^{\geq
    0}(\D,\mm)$.
\end{claim}

Using the Claim, we see that the first projection of the
direct product in the  RHS of the
isomorphism above provides a canonical homomorphism
\beq{can_hom}
\varkappa:\ \Aut(\D,\mm)\to\k^\times.
\eeq

\begin{proof}[Sketch of Proof of Claim]
First one shows, similarly to the Lie algebra case,
that 
$\Ker(\forget_{\Aut})=\eps_{\Aut}(K^\times)$. Further,
we know that $\tau(\Aut(\D)_\J)\sset P$ and
$\Ker(\tau)$ is a prounipotent group.
It follows
that $\tau\ccirc\forget_{\Aut}(\Aut(\D,\mm))\sset P$
and that $\Ker(\tau\ccirc\forget_{\Aut})/\eps_{\Aut}(\k^\times)$
 is a prounipotent group.
The proof is now completed by showing that the
Lie algebra of the latter group equals the
Lie algebra of the group $\Aut^{\geq 1}(\D,\mm)$.

We leave details to the reader.
\end{proof}

Write $a\mto \theta_{\D,\mm}(a)=(\theta_\D(a),\theta_\mm(a))$ 
for the 
Lie algebra homomorphism $\theta_{\D,\mm}: \fp\to\Der(\D,\mm)$ induced
by the group homomorphism $\Theta_{\D,\mm}$. 
There is a diagram
\beq{non}
\xymatrix{
\fp\ar[dr]_<>(0.5){\theta_{\D,\mm}}\ar[r]^<>(0.5){\sigma}&\oh\J\ar[d]^<>(0.5){\varphi_{\D,\mm}}\\
&\Der(\D,\mm)
}
\eeq

It is immediate from definitions that
$\theta_\D=\varphi_\D\ccirc\sigma$. However,
the corresponding diagram involving the module $\mm$
does {\em not} commute; indeed, one has
an equation
\beq{error}
(\varphi_\mm\ccirc\sigma)(a)=\theta_\mm(a)+\half\Tr(a|_\y)\cdot \id_\mm,\qquad\forall a\in\fp.
\eeq
To prove \eqref{error}, note that  for $u\in\D$, we have
$\theta_\D(a)(u)(1_\mm)=\theta_\mm(a)(u(1_\mm))$, by \eqref{dm}.
Hence, using \eqref{sisi} and  Lemma \ref{1m}, we compute
\begin{align*}
(\varphi_{\mm}\ccirc\sigma)(a)(u (1_\mm))&=
(\sigma(a)u)(1_\mm)=[\sigma(a),u](1_\mm)+u(\sigma(a)(1_\mm))\\
&=
\theta_\D(a)(u)(1_\mm)+ u(\half\Tr(a|_\y)\cdot 1_\mm)\\
&=
\theta_\mm(a)(u(1_\mm))+\half\Tr(a|_\y)\cdot u(1_\mm).
\end{align*}

This proves equation \eqref{error} since any element of $\mm$
has the form $u(1_\mm)$ for some  $u\in\D$.\qed

\begin{proof}[Proof of Proposition \ref{Lie-extensions}]
The isomorphism $\varphi_{\D,\mm}: \G^{\geq 1}\iso \Der^{\geq
  1}(\D,\mm)$,
of pronilpotent Lie algebras, can be exponentiated to an isomorphism
$\Phi^{\geq 1}_{\D,\mm}: \GG^{\geq 1}\iso \Aut^{\geq  1}(\D,\mm)$, 
of the corresponding prounipotent groups.
We define a map 
$$
\k^\times_\GG/\{\pm 1\} \times \bigl(\Sigma(P)\ \ltimes\ \GG^{\geq 1}_\J\bigr)\to
\eps_{\Aut}(\k^\times)/\{\pm 1\} \times \bigl(\Theta_{\D,\mm}(P)\ \ltimes\
\Aut^{\geq 1}(\D,\mm)\bigr)
$$
by the formula
\[c\times \bigl(\Sigma(p) \ltimes g\bigr)\ \mto\
\eps_{\Aut}\left({c\cdot\mbox{$\sqrt{\det(p|_\y)}$}}\right)\times \bigl(\Theta_{\D,\mm}(p)\ltimes
\Phi^{\geq 1}_{\D,\mm}(g)\bigr).
\]

It is clear that this map is an isomorphism of proalgebraic groups.
Thanks to Claims \ref{1semidirect} and \ref{2semidirect}, 
the above map  gives an isomorphism
\[ \Phi_{\D,\mm}:\ \GG_\J/\{\pm 1\} \iso \Aut(\D,\mm)/\{\pm 1\} .\]
Moreover, equation \eqref{error} insures that the map
$\Lie\GG_\J\to \Lie\Aut(\D,\mm)$, the Lie algebra homomorphism 
induced by $\Phi_{\D,\mm}$,
is equal to the map $\varphi_{\D,\mm}|_{\GG^{\geq 0}_\J}:\
\GG^{\geq 0}_\J\to \Der^{\geq 0}(\D,\mm)$.
The latter map is a  Lie algebra isomorphism. We conclude
that the pair of maps $\la\Phi_{\D,\mm}, \varphi_{\D,\mm}\ra$
yields an isomorphism  of Harish-Chandra pairs as required
in the statement of Proposition ~\ref{Lie-extensions}.
\end{proof}

\section{Harish-Chandra Torsors}
\label{sec4}
\subsection{}\label{ssec4.1} We will use the notation $\o:=\o_\k$
and write $\T_Y$ for the tangent sheaf of a smooth variety $Y$.

The following definition, that has been used  in \cite{BK}, is due to
Beilinson and Drinfeld \cite{BD}.

\begin{definition} Let $Y$ be a smooth algebraic variety and
$\la G,\g\ra$  a  Harish-Chandra pair.
A
 \emph{transitive Harish-Chandra torsor}  (or \textit{transitive torsor}
for short) over $Y$ is a $G$-torsor $\pp$ over 
$Y$ equipped with a $G$-equivariant Lie algebra 
homomorphism $\g \to H^0(\pp, \T_\pp)$
that extends the  map $\Lie G \to H^0(\pp, \T_\pp)$,
the differential of the $G$-action on $\pp$,
and induces an isomorphism 
$\g \otimes \mathcal{O}_\pp \simeq  \T_\pp$,
of locally free sheaves on $\pp$.
\end{definition}

Let $\a$ be a vector space viewed as an additive algebraic
group. The Lie algebra of this group is identified with $\a$,
so the pair $\la\a,\a\ra$ is a \hc pair.
Further, let
\beq{aext}
1 \to \langle \a,\a \rangle \to \langle \wt{G}, \wt{\frak g}
\rangle \to \langle G, \g \rangle \to 1.
\eeq
be a central extension of Harish-Chandra
pairs, to be denoted $c$. 

In \cite[Proposition 2.7]{BK}, the authors associate
to any
transitive $\langle G, \g\ra$-torsor
$\pp$ on $Y$ a class $\Loc(\pp, c) \in 
\a\o H^2_\dr (Y)$, sometimes also denoted
$\Loc(\pp,\wt{G}, \wt{\frak g})$,  such that the existence of a lift of $\pp$ 
to a transitive torsor $\wt{\pp}$ over 
$\langle \wt{G}, \wt{\frak g}\ra$ 
 is equivalent 
to the vanishing of the class $\Loc(\pp, c)$.

We now recall the construction of  $\Loc(\pp, c)$
since  some  functorial properties of the
construction will be used  later in the paper.
\smallskip

{\sc{Construction:}}\en We start with
the following exact sequence of $\g$-modules:
\beq{ext_a}
0 \to \a \to U_+(\wt\g)/ \big(\a \cdot U_+(\wt\g) \big)
\to U(\g) \to \k \to 0,
\eeq
where $U(-)$ denotes the universal enveloping algebra of a 
Lie algebra and $U_+(-)$ its augmentation ideal. 
Note that the adjoint action of the
\hc pair $\la\wt G,\wt\g\ra$ on itself
factors through  $\la G,\g\ra$.
Therefore,  \eqref{ext_a} is an extension of $\g$-modules.
This extension provides an explicit representative for
the class in $H^2(\g,\a)=\Ext^2_{\g}(\k, \a)$ that corresponds
to the Lie algebra central extension $\a\into\wt\g\onto\g$.

The $G$-action on each term in  \eqref{ext_a} gives an associated
vector bundle on $Y$ corresponding to the $G$-torsor $\pp$.
Moreover, the $\g$-action provides each of these vector bundles with a
flat connection. Further,
the exact sequence  \eqref{ext_a} induces 
an exact sequence of the associated vector bundles, which is compatible
with the connections. The latter exact sequence gives an extension class
in $\Ext^2_{\text{loc syst}}(\mathcal{O}_Y, 
\a \otimes \mathcal{O}_Y)$, where $\Ext^2_{\text{loc syst}}$
denotes the $\Ext$-group in the category of
vector bundles with flat connections, i.e.,
the category of local systems (not necessarily of finite rank,
in general). 
One defines $\Loc(\pp, c)\in \a\otimes H^2_\dr (Y)$
to be the element that corresponds to the extension class
via the canonical isomorphism $\Ext^2_{\text{loc syst}}(\mathcal{O}_Y, 
\a \otimes \mathcal{O}_Y)=\a\otimes H^2_\dr (Y)$.

The following result is immediate
from the above construction of the class $\Loc(\pp, c)$.

\begin{corollary}\label{functorial} Fix a \hc pair $\la G,\g\ra$
and a central extension $c$ as in \eqref{aext}.

\vi Let $f: \la H,{\fra h}\ra\to\la G,\g\ra$ be a morphism of \hc pairs
and let $f^*(c)$ denote the extension of $\la H,{\fra h}\ra$
by $\la\a,\a\ra$ obtained by pull-back of 
\eqref{aext} via $f$. 
Then, in $\a\otimes H^2_\dr (Y) $, we have
$\dis\Loc(\pp,f^*(c))=\Loc(\pp,c)$.

\vii Let $f: \a\to\a'$ be a linear map,
and let
$f_*(c)$ be the extension of  $\la G,\g\ra$
by $\la\a',\a'\ra$ obtained by push-out  of 
\eqref{aext} via $f$.  
Then, we have
$\dis\Loc(\pp,f_*(c))=(f\otimes \id)(\Loc(\pp,c))$,
where $f\otimes\id:\a\otimes H^2_\dr (Y)
\to \a'\otimes H^2_\dr (Y) $ is the map induced by $f$.
\end{corollary}

\subsection{}\label{triv} 
Fix a central extension $c$ of Harish-Chandra pairs
 \beq{triv2}c:
1 \to \langle \k^\times, \k \rangle \to \langle \wt{G}, \wt{\frak g}
\rangle \to \langle G, \g \rangle \to 1.
\eeq
and a splitting $\io: G \to \wt{G}$  of the projection $\wt G\to G$
(the Lie algebra projection $\wt\g\to\g$ is,
however, not assumed to be split, in general).
The splitting induces a group isomorphism $\wt{G} \simeq \k^\times 
\times G$ (that may depend on the choice of ~$\io$). 

In the above setting, to any transitive  Harish-Chandra 
$\la G, \g\ra$-torsor $f: Z\to Y$
one associates a class $\alpha(Z,c,\io)\in H^2(\Om^{\geq1}_Y)$
as follows.

Tensoring the Lie algebra
central extension $\k\into\wt\g\onto\g$ by $\oo_Z$ 
and using the isomorphism $\g\o\oo_Z\cong \T_Z$
yields
an exact sequence
\beq{picard}0\to \oo_Z\too \wt{\g}\o\oo_Z\stackrel{\eta}\too \T_Z\to 0,
\eeq
of $\oo_Z$-modules.  As explained in  \cite[\S1.2.2]{BB}, 
there is a natural Lie algebroid structure
on $\wt{\g}\o~\oo_Z$ such that the corresponding anchor map
is the map $\eta$ in \eqref{picard}.
Further, the group $G$ acts on $\wt{\g}$
and on $Z$, giving the sheaf  $\wt{\g}\o\oo_Z$
a  $G$-equivariant structure. 
Moreover,
the composition of Lie algebra maps 
\[\Lie G\stackrel{{\mathrm d}\io}\too\Lie \wt G\into\wt\g=\wt\g\o 1\into
\wt{\g}\o\oo_Z\] 
 gives a partial splitting of $\eta$.
Using the terminology
of \cite[\S2.1.3]{BB}, the above data gives the sheaf $\wt{\g}\o\oo_Z$ 
the structure of a $G$-equivariant  {\em Picard algebroid} on $Z$.
Applying equivariant descent for Picard algebroids,
as explained in \cite[\S1.8.9]{BB},
one obtains from  $\wt{\g}\o\oo_Z$ a Picard algebroid
$\wt{\g}_Y$ on $Y$. Explicitly, $\wt{\g}_Y$ is the quotient
of $[f_* (\wt{\g}\o\oo_Z)]^G$ by the image of $[f_* (\Lie G \o 
\oo_Z)]^G$. 

We define  $\alpha(Z, c,\io)\in H^2(\Om^{\geq 1}_Y)$ to be 
the
Atiyah class of $\wt{\g}_Y$.

\begin{rem}
\label{io_rem} 
The class $\alpha(Z, c,\io)$ only depends
on the action of $G$ on $\wt{G}$ and the differential of $\io$.
\end{rem}

Let $\pp$ be  $\la \wt{G}, \wt{\g}\ra$-torsor on $Y$.
Given an extension \eqref{triv2}, one obtains a
 $G$-torsor $\k^\times\backslash \pp\to Y$, a push-out
of $\pp$ via the  projection $\wt G\to G$.
Further, given a section $\io: G\to\wt G$, as above,
one obtains a group decomposition $\wt{G} 
\simeq \k^\times\times G\to\k^\times$.
Hence, there is
$\k^\times$-torsor $G\backslash \pp\to Y$, a push-out
of $\pp$ via the first projection $\wt{G} 
\simeq \k^\times\times G\to\k^\times$.

\begin{lemma}\label{bij} The assignment $\pp\mto (\k^\times\backslash \pp,\,
G\backslash \pp)$ yields a bijective map from
the set (of isomorphism classes) of 
transitive $\la \wt{G}, \wt{\g}\ra$-torsors
on $Y$ onto  the set (of isomorphism classes) of
pairs $(Z,L)$, where $Z$ is a transitive $\la G, \g\ra$-torsor
and $L$  is a $\k^\times$-torsor on $Y$ satisfying
 an equation $\alpha(Z, c,\io)=c_1(L)$.
\end{lemma}

\begin{proof} 
Given any $\wt G$-torsor $\pp$,
put $L:=G\backslash\pp,\
Z:=\k^\times\backslash \pp$, and let
the group  $\k^\times \times G$ act on
$L\times_YZ$ by $t\times g:l\times z\mto t(l)\times g(z)$
(i.e. $\k^\times$ acts on the $L$ factor and $G$ on the $Z$ factor).
Then the map $p\mto (G p)\times (\k^\times p)$ gives
a $\wt G$-equivariant isomorphism $\pp\iso L\times_YZ$.
Conversely, given a $G$-torsor $Z$ and a $\k^\times$-torsor
$L$, put $\pp:=Z\times_YL$ and   let $\wt G$ act on $\pp$ as above.
It is clear that this
makes $\pp\to Y$ a $\wt G$-torsor. 

Thus, proving the lemma amounts to
showing that the equation $c_1(L)=$ $\alpha(Z,c,\io)$
insures that there exists a $G$-equivariant Lie algebroid structure on
$\wt \g\o\oo_{\pp}$ which is compatible with
 the second projection
$\pr: \pp=L\times_YZ\to Z$ (in the sense that the
projection of $\wt \g$ onto its quotient mod $\k$ agrees with the 
differential $\T_\pp \to \pr^*\T_Z$).  
It suffices to
construct the corresponding anchor map $\eta:
\wt {\g}\o\oo_\pp\to\T_\pp$.
To this end, let
$\mathsf{At}(\pp/Z)=(\pr_*\T_Z)^{\kt}$ be the
 Atiyah algebra of  the $\kt$-torsor
$\pr: \pp\to Z$.
Since this torsor $\pp\to Z$ is a pull-back of the torsor
$L\to Y$, for the Atiyah classes we have
\[At(\pp/Z)=\pr^*c_1(L)=\pr^*\alpha(Z, c,\io).\]
The class on the right equals, by construction
of $\alpha(Z, c,\io)$, the class of the extension \eqref{picard}
(we use the fact that Atiyah classes on $Y$ are in bijective
correspondence with $G$-equivariant Atiyah classes on $Z$).
We deduce that there is an isomorphism $\wt\g\o\oo_Z\iso 
\mathsf{At}(\pp/Z)$,
of $G$-equivariant Lie algebroids on $Z$. We define
 $\eta$ to be this isomorphism.
\end{proof}

\subsection{} We return to the setup of \S\ref{sec3.2}.

We put $\bar\GG_\J=\GG_\J/\k_\GG^\times$,
resp. $\bar\G_\J=\G_\J/\k_\J$.
From Claim \ref{1semidirect}
we deduce
a natural isomorphism
$\ \bar\GG_\J\cong \Sigma(P)\ltimes
\GG_\J^{\geq1}$ and hence a direct product decomposition
\beq{dec_J}
\GG_\J\cong \k^\times\times 
\bigl(\Sigma(P)\ltimes
\GG_\J^{\geq1}\bigr)\cong\k^\times\times
\bar\GG_\J.
\eeq
The above decomposition yields
a splitting $\io_\J: \bar\GG_\J\to \GG_\J$ of the
canonical projection $ \GG_\J\to \bar\GG_\J$.
Observe next  the assumptions 
formulated at the beginning
of section \ref{triv} hold in the case where
$G=\bar\GG_\J,\ \wt\g=\G_\J$, and
$\io_\J$ is the natural imbeding.

We  also consider a \hc pair
\[\langle \oa, \od \rangle\
:=\
\langle \Aut(\D, \mm)/\eps_{\Aut}(\k^\times),\, \Der(\D, \mm)/\eps_{\Der}(\k) \rangle.
\]
The isomorphism of Claim  \ref{2semidirect},
yields a  direct product decomposition
\beq{dec_der}
\Aut(\D,\mm)\cong\k^\times\times\oa.
\eeq
and we denote by $\io_{\Der}: \oa \to \Aut(\D,\mm)$ the 
homomorphism that results from the imbedding of the second
  factor. Again, 
the data $G=\oa,\ \wt\g=\Der(\D,\mm)$,
and  $\io_{\Der}$,
satisfy the assumptions 
formulated at the beginning
of  section \ref{triv}.

The isomorphism $\la\Phi_{\D,\mm},\varphi_{\D,\mm}\ra$, of Proposition \ref{Lie-extensions},
induces an  isomorphism $\la\bar\Phi_{\D,\mm},\bar\varphi_{\D,\mm}\ra$,
of \hc pairs,
in the following  diagram
\beq{barphi}
\xymatrix{
\la K^\times/\k^\times, K/\k\ra\
 \ar@{^{(}->}[r]\ar@{=}[d]^<>(0.5){\id}
&\la\bar\GG_\J,\bar\G_\J\ra\ar@{->>}[r]\ar[d]_<>(0.5){\cong}
^<>(0.5){\la\bar\Phi_{\D,\mm},\bar\varphi_{\D,\mm}\ra}
&
\la\Aut(\D)_\J, \Der(\D)_\J\ra\ar@{=}[d]^<>(0.5){\id}\\
\la K^\times/\k^\times, K/\k\ra\
\ar@{^{(}->}[r]&
\la\oa,\od\ra\ar@{->>}[r]&
\la\Aut(\D)_\J, \Der(\D)_\J\ra 
}
\eeq
of central extensions of Harish-Chandra pairs.

\begin{rem}
We note that because of noncommutativity of
diagram \eqref{non}, the
map $\varphi_{\D,\mm}$ does not respect
the Lie algebra decompositions resulting from decompositions \eqref{dec_J}
and \eqref{dec_der}, respectively.
Therefore, the pair of maps
$(\bar\Phi_{\D,\mm}, \varphi_{\D,\mm}):\
\la\bar\GG_\J,\G_\J\ra\to
\la\oa, \Der(\D,\mm)\ra$
does {\em not} form a morphism of \hc pairs.
\erem

Given a  $\la\oa,\od\ra$-torsor 
$Z$, let $\bar\Phi_*Z$ denote
the $\la\bar\GG_\J,\bar\G_\J\ra$-torsor  obtained
by transporting the torsor structure via
the isomorphism $\la\bar\Phi_{\D,\mm}, \bar\varphi_{\D,\mm}\ra$,
of \hc pairs. We also denote by $c_{\Der}$, resp. $c_\J$, the 
 natural central extension of $\la\oa,\od\ra$, resp. 
$\la\bar\GG_\J,\bar\G_\J\ra$, by $\la \k^\times, \k \ra$.

\begin{proposition}\label{jdm}
For any  $\la\oa,\od\ra$-torsor
$Z$, on $Y$, we have in $H^2(\Om_Y^{\geq1})$ 
\[
\alpha(Z, c_{\Der},\io_{\Der}) - 
\alpha(\bar\Phi_*Z, c_\J,\io_\J)= \half c_1(L_Z),
\]
where $L_Z$ is the $\k^\times$-torsor on $Y$ induced from $Z$ via
the composition of homomorphisms:
\[
\oa \to \Aut(\D)_\J \to \Sigma(P) \to \k^\times 
\]
where the last arrow is the character $\Sigma(p) \mapsto \det(p|_\y)$. 
\end{proposition}
\begin{proof} 
In the notation of \ref{triv} we observe that for both central extensions
$c_\J$, $c_{\Der}$ the Lie algebra part $\wt\g$ can be identified 
with $\oh\J$. Although  this identification of Lie algebras 
does not extend to a 
homomorphism on the group parts, only on their quotients by the
image of $\{\pm 1\}$, for the construction of the $\alpha$ classes we
only need the adjoint action of $G$ on $\wt\g$, which does match
for the two extensions, and the differentials 
$d\io: \g \to \wt\g$ of the splitting maps $\io: G \to
\wt{G}$, which are  \textit{different} for $c_{\Der}$ and
$c_\J$, respectively. 

 In fact, identifying $\od$ with $\G_\J$ and applying 
\eqref{error} we see that $(d\io_{\J} - d \io_{\Der})$ is the
composition 
$$
 \Lie(\oa) \to \Lie(\Aut(\D)_\J) \to \fp \to \k
\subset \G_\J
$$
where $\fp \to \k$ is given by $\half\Tr(a|_\y)$. We will 
denote this composition also by $\half\Tr(a|_\y)$. 
Therefore we see that the classes $\alpha(Z, c_{\Der},\io_{\Der})$
and $\alpha(\bar\Phi_*Z, c_\J,\io_\J)$ arise from the 
same Lie algebroid $\G_\J \otimes \oo_Z$ but equipped with 
different $\oa$-equivariant structures. More precisely, the 
group action on the sheaf is the same in both cases but
the partial connection along the fibers of the projection $Z \to Y$ 
differs by the map $\half\Tr: \Lie(\oa)\to \k$ described above. 

Using the group structure on the set of isomorphism classes
of Lie algebroids, it suffices to check that $\half c_1(L_Z)$
is the class of the trivial Lie algebroid $\oo_Z \oplus 
\T_Z$ with the equivariant structure in which the 
canonical embedding 
$$
\Lie(\oa) \otimes \oo_Z \hookrightarrow
\od \otimes \oo_Z \simeq \T_Z \hookrightarrow 
\oo_Z \oplus \T_Z
$$
is adjusted by $\half \Tr \otimes \oo_Z$. Furthermore since the group
of isomorphism classes of Lie algebroids is a vector space 
over a field of characteristic zero, it suffices to show the 
statement without both factors $\half$. Then the trace 
map $\Tr: \Lie(\oa)\to \k$ integrates to the group 
homomorphism $\oa \to \k^\times$ described in the 
statement of the proposition. Applying equivariant descent 
with respect to the kernel $\mathcal{U}$ of 
$\oa \to \k^\times$ we reduce 
the statement to the assertion that $c_1(L_Z)$ is the 
class of the Atiyah algebra of $\mathcal{U}\backslash Z = L_Z$, which 
holds by definition of the first Chern class 
$c_z(L_Z)$ with values in  $H^2(Y, \Om^{\geq 1}_Y)$.
\end{proof}

\section{Torsors associated with a quantization}
\label{quant_sec}
\subsection{}\label{pp}

Let $X$ be a smooth symplectic variety and for any point
$x\in X$ let $\wh\oo_x$ denote the completion of the local
ring at $x$. A choice of a formal symplectic coordinate system
at $x$ is equivalent to a choice of a topological
$\k[[\hb]]$-algebra isomorphism 
$\eta: \wh\oo_x\iso \A$ of Poisson algebras.
Composing $\eta$ with an automorphism of $\A$ yields
another  isomorphism $\wh\oo_x\iso \A$.
This shows that  the  pairs $(x,\eta)$, as above, form the set of
(closed) points of a transitive Harish-Chandra $\langle \Aut(\ca A), \Der(\ca
A)\rangle$-torsor $\pp_X$, on ~$X$ (both automorphisms 
are derivations are assumed to preserve the Poisson structure 
on $\ca A$).

Next, let $\oo_\hb$ be a formal  quantization 
 of $\oo_X$ and let  $\mathcal{O}_{x,\hb}$
denote  the completion  of $\mathcal{O}_\hb $ at a point $x\in X$.
The algebra  $\mathcal{O}_{x,\hb}$ is
isomorphic to $\D$ as a topological $\k[[\hb]]$-algebra.
Furthermore, the  pairs $(x,\eta_\hb)$,
where $x\in X$ and $\eta_\hb: \D\iso \mathcal{O}_{x,\hb}$ is
an isomorphism of topological $\k[[\hb]]$-algebras, form the set of
(closed) points of a transitive Harish-Chandra $\langle \Aut(\D),
\Der(\D) \rangle$-torsor
$\pp_\hb$, on $X$.

The algebra projection $\D\onto\D/\hb\D=\A$
induces  a canonical projection
$\langle \Aut(\D), \Der(\D) \rangle\onto\langle \Aut(\ca A), \Der(\ca
A)\rangle$, of  Harish-Chandra pairs.
According to an observation of \cite{BK}, a choice of deformation quantization 
$\mathcal{O}_\hb$, of $\oo_X$, is equivalent
to a choice of lift of the  torsor $\pp_X$ over 
$\langle \Aut(\ca A), \Der(\ca A)\rangle$ to a transitive Harish-Chandra torsor
$\pp_\hb$ over $\langle \Aut(\D), \Der(\D) \rangle$ (we just choose an
identification of the completion $\mathcal{O}_{x,\hb}$, of $\mathcal{O}_\hb $ at $x$, with 
$\D$).

From now on, we fix a quantization $\oo_\hb$
and  an associated torsor $\pp_\hb$ as above.
Recall that the period map assigns to
the quantization $\oo_\hb$ a class
$\operatorname{per}(\oo_\hb)$,
of the form \eqref{per}. This class
is defined in terms of the torsor  $\pp_\hb$
as follows.

First, one introduces a proalgebraic group
$\ \dis\GG^+:=\k\times \big(Sp(\v)\ltimes\GG^{\geq1}\big)$,
which  is almost isomorphic to the group $\GG$
except that the multiplicative group $\k^\times_\GG$,
in the center of $\GG$,  is
replaced by a copy of the additive group $\k$. 
There is, in fact, a copy of the additive  group 
$\k[[\hb]]$ contained in the center of $\GG^+$.
The imbedding $\k[[\hb]]\into\GG^+$ is defined,
via the  natural group decomposition $\k[[\hb]]=\k\times
\hb\k[[\hb]]$, 
to be a 
 cartesian product of the  imbedding $\k\into\GG^+$, as the first factor,
and the composition $\hb\k[[\hb]]\iso 1+\hb\k[[\hb]]\into \GG^{\geq1}$,
where the first map is the exponential map
and the second map was discussed after formula \eqref{sisi}.
Further, using that the morphism $\la\Phi_\D,\varphi_\D\ra$
vanishes on the subpair $\la\k^\times_\GG,\k_\GG\ra\sset
\la\GG,\G\ra$ one constructs an
`additive counterpart' of   \eqref{aut}, a central extension
of \hc pairs of the form:
\beq{bk_ext}
1 \to \langle \k[[\hb]], \k[[\hb]] \rangle \too 
\langle \GG^+,  \oh\D \rangle \stackrel{\phi^+}\too 
\langle \Aut(\D), \Der(\D) \rangle \to 1,
\eeq
of \hc pairs.
Then, one defines $\operatorname{per}(\oo_\hb):=\Loc(\pp_\hb,
\GG^+,  \oh\D)$,
the obstruction class to lifting the torsor  $\pp_\hb$,
over $\langle \Aut(\D), \Der(\D) \rangle$,
to a transitive torsor over the \hc pair
$\langle \GG^+,  \oh\D \rangle$.

\subsection{}\label{pj}

Let $i_Y :Y \into X$ be a closed imbedding of a smooth Lagrangian subvariety
and $I_Y$ the ideal sheaf of $Y$.
Let $\J_Y$ be  the preimage of the ideal $I_Y$
under the natural projection $\oo_\hb\onto\oo_X$.
For $y\in Y$,  let $I_{Y,y}$, resp.
$\J_y$, denote the completion of $I_Y$, resp. $\J_Y$, at $y$.
It is clear that one can choose
a $\k[[\hb]]$-algebra  isomorphism $\eta: \mathcal{O}_{y,\hb}\cong \D$
in such a way that $\eta(\J_y)=\J$.
Moreover, the pairs  $(y,\eta)$ form the set of
(closed) points of a transitive   torsor 
$\pp_\J$ over the
\hc pair $\langle \Aut(\D)_{\ca J}, \Der(\D)_{\ca J} \rangle$
(for transitivity, observe that the bracket with $\frac{1}{\hb}{\ca J}$
preserves $\ca J$ and that the conormal bundle to $Y$ is 
identified with its tangent bundle, due to the Lagrangian condition).

Write $q^\times: K^\times\to K^\times/\k^\times$,
resp. $q^+: K\to K/\k$, for the natural projections. Let
\beq{K/k}
1\to \la K^\times/\k^\times, K/\k\ra \to \la\bar\GG_\J,\bar\G_\J\ra
\to \la\Aut(\D)_\J,\Der(\D)_\J\ra\to1
\eeq
be
a push-out of the  extension \eqref{aut}  via the morphism
$\la q,q^+\ra: \la K^\times, K\ra \to \la K^\times/\k^\times, K/\k\ra$.
Associated with the  central extension \eqref{K/k} one has
the obstruction class
$\Loc(\pp_\J, \bar\GG_\J,\bar\G_\J)$ 
for lifting the torsor $\pp_\J$ to a transitive torsor over
$\la\bar\GG_\J,\bar\G_\J\ra$.

\begin{lemma}\label{loc_j} 
In $H^2(Y)[[\hb]]$, we have
\[ \Loc(\pp_\J,
\bar\GG_\J,\bar\G_\J)=i_Y^*(\hb^2\,\om_2(\ohb)+\hb^3\,\om_3(\ohb)+\ldots))\red{.}\]

\end{lemma}
\begin{proof} First of all, it is immediate to see
that the $\la\Aut(\D),\Der(\D)\ra$-torsor
$i_Y^*\pp_\hb$, the restriction of the torsor $\pp_\hb$
to the subvariety $Y$, is induced from the torsor
$\pp_\J$ via the imbedding
$\la\Aut(\D)_\J,\Der(\D)_\J\ra\into \la\Aut(\D),\Der(\D)\ra$.

Let $c$ be  a pull-back of
the central extension \eqref{bk_ext} with respect to this embedding
and let
$\la\GG_\J^+,\oh\D_\J\ra$ be the
preimage of the \hc pair $\la\Aut(\D)_\J,$ $\Der(\D)_\J\ra$
under the morphism $\phi^+$ in \eqref{bk_ext}.
Applying Corollary \ref{functorial}(i) 
we deduce that the  obstruction 
for lifting the torsor $\pp_\J$ to a 
$\la\GG_\J^+,\oh\D_\J\ra$-torsor
is equal to $i_Y^*([\om]+\hb\,\om_1+\hb^2\,\om_2+\hb^3\,\om_3+\ldots)$.
The variety $Y$ being Lagrangian, we have $i_Y^*([\om])=0$.

Note that the Lie
algebra $\oh\D_\J$ breaks up into a direct sum
$\oh\k\oplus \oh\J$.
Let
\[1\to \la K/\k, K/\k\ra \to \la\GG_\J^+/\k,\oh\J\ra
\to \la\Aut(\D)_\J,\Der(\D)_\J\ra\to1,
\]
be the push-out of the resulting  extension   via the morphism
$\la q^+,q^+\ra: \la K, K\ra$ $\to \la K/\k, K/\k\ra$.
Applying part (ii) of  Corollary  \ref{functorial}
we conclude  that the  obstruction
for lifting the torsor $\pp_\J$ to a torsor over
$\la\GG_\J^+/\k,\oh\J\ra$
is equal to $i_Y^*(\hb^2\,\om_2+\hb^3\,\om_3+\ldots)$
(i.e. the term $i_Y^*(\hb\,\om_1)$ disappears when we
take the quotient by $\la \k, \k \ra$).

On the other hand, one has natural isomorphisms
$\GG_\J^+/\k\cong \Sigma(P)\ltimes \GG_\J^{\geq1}\cong
\GG_\J/\k_\GG^\times$, cf. Claim \ref{1semidirect}.
Furthermore, it is straightforward to see by
comparing the constructions that the central
extension in the displayed formula above
is isomorphic to the one in \eqref{K/k};
the isomorphism being induced by
the exponential map $\exp: K/\k\iso  K^\times/\k^\times$.
The isomorphism of extensions
implies an equality of the corresponding
obstruction classes, and the result follows
from the conclusion of the preceeding paragraph.
\end{proof}

\subsection{}\label{J-sec} In this subsection, we are going to 
associate to the quantization $\oo_\hb$ and the Lagrangian subvariety
$Y\sset X$
a Picard algebroid of the form 
\eqref{pic_J}.

To this  end,
let $\J'_Y$
be  the preimage of the ideal  $I^2_Y$
under the  projection $\oo_\hb\onto\oo_X$,
and write $\J_Y^2:=(\J_Y)^2$.
It is clear that one has inclusions $\J_Y^2\sset\J'_Y\sset\J_Y$.

\begin{lemma}\label{can}
There are canonical isomorphisms
\[\J_Y/\J_Y'\cong\T_Y,\quad
\J_Y'/\J_Y^2\cong\oo_Y,\quad\text{and}\quad
\J_Y/\J_Y^2\cong\Tor_1^{\oo_\hb}(\oo_Y,\oo_Y).
\]
\end{lemma}
\begin{proof} 
By definition, the projection
  $\pr: \oo_\hb\onto\oo_X$ induces an isomorphism
$\J_Y/\J_Y'\iso I_Y/I_Y^2$. Further, the symplectic form on $X$ 
provides an isomorphism between  $I_Y/I_Y^2$, the conormal sheaf 
to $Y$,
and the tangent sheaf $\T_Y$. The first isomorphism
of the lemma follows.

To prove the second isomorphism, note that we have
$\pr(\J_Y^2)=I_Y^2=\pr(\J_Y')$.
It follows that the natural imbedding
$\hb \oo_\hb\into\J_Y'$ induces an isomorphism
$\hb \oo_\hb/(\J_Y^2\cap \hb\oo_\hb)\iso \J_Y'/\J_Y$.
Clearly, one has an inclusion
$\hb\J_Y\sset \J_Y^2\cap \hb\oo_\hb$. Furthermore,
for any  $y\in Y$, an explicit computation
in local coordinates shows that the inclusion
$\hb\J_{Y,y}\into\J^2_{Y,y}\cap \hb\oo_{\hb,y}$
is, in fact, an equality.
It follows that $\J_Y^2\cap \hb\oo_\hb=\hb\J_Y$.
Thus, we deduce a chain of isomorphisms
\begin{multline*}
\J_Y'/\J_Y\ \cong\ \hb \oo_\hb/(\J_Y^2\cap \hb\oo_\hb)\ \cong\ 
\hb\oo_\hb/\hb\J_Y\\ 
\cong\ \hb(\oo_\hb/\J_Y)\ \cong\ \hb(\oo_X/I_Y)\ \cong\ \hb\oo_Y.
\end{multline*}
The second isomorphism of the lemma follows.

To prove the third isomorphism we use
the long exact sequence of $\Tor$-sheaves associated with the
short exact sequence $\J_Y\into\oo_\hb\onto\oo_Y$.
A final part of that long exact sequence reads
{\begin{multline*}
\ldots\too \Tor^{\oo_\hb}_1(\oo_Y,\oo_\hb)\too
\Tor^{\oo_\hb}_1(\oo_Y,\oo_Y)\stackrel{a}\too
\oo_Y\bo_{\oo_\hb}\J_Y\\
\too\oo_Y\bo_{\oo_\hb}\oo_\hb\stackrel{b}\too
\oo_Y\bo_{\oo_\hb}\oo_Y\too0.
\end{multline*}}
We have $\Tor^{\oo_\hb}_1(\oo_Y,\oo_\hb)=0$
and the map $b$ above is, essentially,
the identity map $\oo_Y\to\oo_Y$.
It follows from the exact sequence that the map $a$ is an isomorphism.
It remains to note that one has isomorphisms
$\oo_Y\bo_{\oo_\hb}\J_Y=(\oo_\hb/\J_Y)\bo_{\oo_\hb}\J_Y=
\J_Y/\J_Y^2$.
\end{proof}

It is easy to see using $\{I_Y, I_Y\} \subset I_Y$ that the bracket 
$\oo_\hb\times \oo_\hb\to \oo_\hb,\ a\times b \mto \oh(ab-ba)$
descends to
 a well-defined Lie bracket on
$\J_Y/\J_Y^2$, resp.  $\J_Y'/\J_Y^2$ and $\J_Y/\J_Y'$.
Furthermore, the bracket on $\J_Y/\J_Y'$ goes, via the
first isomorphism of Lemma \ref{can},
to the commutator of vector fields.

Now, there is an obvious short exact sequence
\beq{JJ}
0\to\J_Y'/\J_Y^2\to\J_Y/\J_Y^2\to\J_Y/\J_Y'\to0.
\eeq
All the maps in this sequence respect the brackets
and
the image of the element
$\hb\in\J_Y'/\J_Y^2$ is contained in the center of
the Lie algebra $\J_Y/\J_Y^2$.
Thus, we see from  Lemma \ref{can} that our
short exact sequence takes the form of
the extension in \eqref{pic_J}.

\begin{rem}  It is not difficult to see that
the extension \eqref{pic_J} may be interpreted in a natural way as
a  short exact sequence of the form
$$0\to \oo_Y\o_{\oo_\hb}\Tor^{\oo_\hb}_1(\oo_X,\oo_X)\to
\Tor^{\oo_\hb}_1(\oo_Y,\oo_Y)\to
\Tor^{\oo_X}_1(\oo_Y,\oo_Y)\to0,
$$
where the tensor product in the first term involves
the $\oo_\hb$-module structure on $\Tor^{\oo_\hb}_1(\oo_X,-)$ induced
by the $\oo_\hb$-action  on $\oo_X=\oo_\hb/\hb\oo_\hb$
on the left.
The above exact sequence is 
a noncommutative version of the Jacobi-Zariski sequence,
cf. \cite[\S 3.5.5]{Lo}, associated with the 
 algebra homomorphisms $\oo_\hb\onto\oo_X\onto\oo_Y$.
\erem

We let
$At(\oo_\hb,Y)\in H^2(\Om_Y^{\geq 1})$
be the Atiyah class of the extension  ~\eqref{pic_J},
equivalently, of  the extension  ~\eqref{JJ}.

\begin{rem}
There is an alternative construction of the class
$At(\oo_\hb,Y)$ in terms of Cech cocycles as follows.

Locally in the Zariski topology we can write 
$\mathcal{O}_\hb /\hb^3 \mathcal{O}_\hb $ as $\mathcal{O}_X + 
\hb \mathcal{O}_X + \hb^2 \mathcal{O}_X$ and 
$\J_Y = I_Y + \hb \mathcal{O}_X + \hb^2 \mathcal{O}_X$, 
$\J'_Y = I^2_Y + \hb \mathcal{O}_X + \hb^2 \mathcal{O}_X$.
On an open subset $X_i$ 
the truncated product looks like
$$
a * b = ab + \hb \alpha_1^i (a, b) + \hb^2 \alpha_2^i(a, b)
$$
while on double intersections the two direct sum splittings are related by the 
$\k[[\hb]]$ linear map 
$$
a \mapsto a + \hb \beta_1^{ij}(a) + \hb^2 \beta_2^{ij}(a).
$$
It follows from the standard associativity equations on the $*$ product that
antisymmetrizing $\alpha_2^i$, then choosing $a, b$ only from the ideal of
functions vanishing on $Y$, and finally restricting to $Y$ we obtain, due
to $N^* \simeq \T_Y$, a closed 2-form $\eta^i$ on $Y_i = X_i \cap Y$.
On the double intersections, since the transition functions $\beta_1^{ij}$ agree
with products and since by assumption $\alpha_1^i (a, b) = \frac{1}{2}P(da, db)$, 
we can conclude that each $\beta_1^{ij}$ is a derivation, i.e. induced by a vector
field on $X_i$. Projecting its restriction to $Y$ on the normal bundle $N$ and
then using $N \simeq \Omega^1_Y$ we obtain 1-forms $\xi^{ij}$ on 
$Y_i \cap Y_j$, such that $\eta^i|_{Y_i \cap Y_j}  - \eta^j|_{Y_i \cap Y_j} = 
d \xi^{ij}$. Then the collection $\eta^i, \xi^{ij}$ defines a class in $H^2$ of the
truncated de Rham complex (i.e. the 1-forms $\xi^{ij}$ satisfy the cocycle condition
on triple intersections, rather than up to a differential of function, since the 
cocycle condition holds for the original vector fields $\beta^{ij}$). 
\erem
\medskip

We now consider the setting of \S\ref{triv} in a special case
where 
 $\la G, \mathfrak{g} \ra = \la \Aut(\D)_\J,  
\Der(\D)_\J \ra$ 
and $\la \wt{G}, \wt{\mathfrak{g}} \ra  = 
\la \GG_\J/(1 +  \hb K) , \G_\J/\hb K  \ra$. 
We have a natural extension $\wt{c}$
as in \eqref{triv2}, and also a section $\wt{\io}: G \to \wt{G}$
that comes from the direct product decomposition \eqref{dec_J}.
Thus, the construction of  \S\ref{triv}
associates to this data  a class
$\alpha\big(\pp_\J, \wt{c},\wt{\io}\big)\in
H^2(\Om_Y^{\geq1})$.
\begin{lemma}\label{alat}
\vi In $H^2(\Om_Y^{\geq1})$, one has an equality 
$\alpha\big(\pp_\J, \wt{c},\wt{\io}\big)=
At(\oo_\hb,Y).$
\vii The canonical morphism $H^2(\Om_Y^{\geq1})\to H^2_\dr(Y)$
sends $At(\oo_\hb,Y)$ to $i^*_Y(\om_1(\oo_\hb))$.
\end{lemma}
\begin{proof} By definition, the class $\alpha\big(\pp_\J, \wt{c},\wt{\io}\big)$ is the class of 
the equivariant descent of the Lie algebroid 
$\wt{\mathfrak{g}} \otimes \oo_{\pp_\J}$ and we need
to identify this  with the  Atiyah algebra $\J_Y/\J_Y^2$
(as Lie algebroids on $Y$). Instead, we 
can pull back the latter algebra to $\pp_\J$ and identify the 
pullback with the quotient of 
$\wt{\mathfrak{g}} \otimes \oo_{\pp_\J}$ by 
$d\wt{\io}(\Lie(G)) \otimes \oo_{\pp_\J}$. 
But by definition of $\pp_\J$ at every (closed) point of this
torsor the completion of $\J_Y$ is identified with $\J$ and
$$
\J/\J^2
\simeq (\oh \J \mod \hb K) / d\wt{\io} (\Lie(\Aut(\D)_\J)),
$$ 
as required. This finishes the proof of 
\vi.

For \vii recall that the class in $H^2_\dr(Y)$ is represented by 
a sequence of flat bundles on $Y$ induced by $\pp_\J$ from 
the sequence \eqref{ext_a} with $\mathfrak{a}=\k$. On the other 
hand, for any Atiyah algebra $\mathcal{O}_Y \to \mathcal{L}
\to \T_Y$ the image of its class in $H^2_\dr(Y)$ is represented 
by the  extension
$$
0\to\oo_Y \to U_+ (\mathcal{L})/\oo_Y \cdot U_+(\mathcal{L})
\to \D_Y \to \oo_Y\to0
$$
where $\D_Y$ is the sheaf of algebraic differential operators on $Y$ and
the last arrow sends an operator to its value on the constant function 
$1$.  If the Atiyah algebra in question is the equivariant 
descent of $\wt{\g} \otimes \oo_{\pp_\J}$ then the above long 
extension is obtained from a sequence very similar to 
\eqref{ext_a}, except that the two middle terms are replaced by their
quotients by the ideal generated by the image of $\Lie(G)$. But it is
easy to check that these two ideals are isomorphic, hence the class in 
$\Ext^2$ is the same, as required. 
\end{proof}

\section{Proof of the main result}
\subsection{}
We keep the notation of the previous section.
In particular, we have the obstruction class 
 $\per(\oo_\hb)=[\om]+\hb\,\om_1(\oo_\hb)+\hb^2\,\om_2(\oo_\hb)+\ldots$,
 associated with the extension
\eqref{bk_ext}.

\begin{lemma}
\label{quantizelift}
A choice of  line bundle $L$ on $Y$ and its deformation 
quantization $L_\hb $ is equivalent to the choice of a lift 
of the torsor $\pp_\J$ to a transitive Harish-Chandra torsor $\pp_{\D,\mm}$ over
$\langle \Aut(\D, \mm), \Der(\D,\mm)\rangle$.
\end{lemma}
\begin{proof} For $y\in Y$, let $\eta: \oo_{y,\hb}\cong \D$
be an isomorphism such that $\eta(\J_y)=\J$.
Given a line bundle $L$ and its quantization  $L_\hb $,  let
$L_y:=\oo_{Y,y}\o_{\oo_Y}L$, resp.
$L_{y,\hb}:=\oo_{y,\hb}\o_{\oo_\hb}L_\hb $.
A choice
of local section of $L$ at $y$ provides
an isomorphism $L_y\cong  \oo_{Y,y}$.
We obtain a chain of
 isomorphisms 
\[L_{y,\hb}/\hb\,L_{y,\hb}\cong 
L_y\cong \oo_{Y,y}\cong
\oo_{y,\hb}/(\hb\,\oo_{y,\hb}+\J_y)\cong
\D/\J\cong\mm/\hb\,\mm,
\]
where the fourth isomorphism is induced by
the isomorphism $\eta$.
Thus, applying Lemma \ref{wellknown}, we deduce
that the $\D$-module $\eta^*L_{y,\hb}$, obtained from $L_{y,\hb}$ by 
transporting the module structure via $\eta$, is
isomorphic to $\mm$. Various choices of an isomorphism
$\eta^*L_{y,\hb}\cong\mm$ for all $y\in Y$ give
the required lift of the torsor $\pp_\J$ to a transitive Harish-Chandra torsor $\pp_{\D,\mm}$ over
$\langle \Aut(\D, \mm), \Der(\D,\mm)\rangle$.

In the opposite direction, let  $\pp_{\D,\mm}$ a lift of $\pp_\J$.
Then $\Aut(\D, \mm)$ acts on $\mm$ and, therefore, one has
a vector bundle $\mm_\pp$ associated to
the $\Aut(\D, \mm)$-module $\mm$ and  $\pp_{\D,\mm}$,
viewed as an $\Aut(\D, \mm)$-torsor.
Moreover, the Lie algebra action of $\Lie(\Aut(\D, \mm))$
extends to the action of the full algebra 
 $\Der(D, M) \simeq \G_\J$ which implies that
the bundle $\mm_\pp$ admits a flat algebraic connection.
Now $L_\hb $ may be recovered as the sheaf of flat sections with 
respect to this connection.

Finally, we note that the (non-quantized) line bundle $L$ 
may also be recovered from $\pp_{\D,\mm}$.
Specifically, one has an isomorphism
$L\cong \kt\o_\varkappa\pp_{\D,\mm}$,
of $\k^\times$-torsors on $Y$,
where
$\kt\o_\varkappa\pp_{\D,\mm}$
denotes the push-out of the torsor $\pp_{\D,\mm}$
via the canonical homomorphism $\varkappa: \Aut(\D, \mm) \to \kt $,
see \eqref{can_hom}.
\end{proof}

By Lemma \ref{quantizelift} a choice of \blue{a} quantized line bundle $L_\hb $ on $Y$ is 
equivalent to a choice of the following data:
\begin{align}
&\begin{array}{l}\bullet\ \ \text{{A lift  of the $\langle
  \Aut(\D)_\J,\Der(\D)_\J\ra$-torsor
$\pp_\J$ to a transitive}}\\
\text{{\quad torsor $\bar\pp_{\D,\mm}$ over $\langle
  \oa,\od\rangle$.}}
\end{array}\label{barp}\\
&\begin{array}{l}\bullet\ \ \text{{A lift  of   $\bar\pp_{\D,\mm}$ to a transitive 
$\langle \Aut(\D, \mm), \Der(\D, \mm)\ra$-torsor}}\\
\text{{\quad  $\pp_{\D,\mm}$
such that one has an isomorphism $\ L\cong \kt\o_\varkappa\pp_{\D,\mm}$.}}
\end{array}\label{barpp}
\end{align}

\begin{lemma}\label{step1} The existence of a lift $\bar\pp_{\D,\mm}$,
as in \eqref{barp}, is equivalent to an equation
$\ i_Y^*(\hb^2\,\om_2(\oo_\hb)+\hb^3\,\om_3(\oo_\hb)+\ldots)=0$ in $H^2(Y)[[\hb]]$.
\end{lemma}
\begin{proof} Diagram \eqref{barphi} provides an isomorphism
of central extensions of \hc pairs.
Therefore,  the torsor $\pp_\J$ can be lifted
to a transitive $\la\oa,\od\ra$-torsor if and only
if it can be lifted
to a transitive $\la\bar\GG_\J,\bar\G_\J\ra$-torsor.
The latter holds if and only
if the class $\Loc(\pp_\J,\bar\GG_\J,\bar\G_\J)$
vanishes, see \S\ref{ssec4.1}. The result now follows from
Lemma \ref{loc_j}.
\end{proof}

\subsection{Proof of Theorem \ref{maintheorem}} 
From now on, we assume that the equation
$\ i_Y^*(\hb^2\,\om_2(\oo_\hb)+\hb^3\,\om_3(\oo_\hb)+\ldots)=0$ holds
and hence  there is a torsor  $\bar\pp_{\D,\mm}$,
over $\la\oa,\od\ra$,  as in
\eqref{barp}.

\begin{lemma} The  $\la\oa,\od\ra$-torsor
$\bar\pp_{\D,\mm}$ can be lifted to a 
$\la\Aut(\D,\mm), \Der(\D,\mm)\ra$-torsor
$\pp_{\D,\mm}$, as in \eqref{barpp},  if and only if 
one has: $c_1(L)$ $=\half c_1(K_Y)+
At(\oo_\hb,Y)$.
\end{lemma}

\begin{proof} By \eqref{dec_der} any lift $\pp_{\D,\mm}$
must be isomorphic as an $\Aut(\D, \M)$-torsor to 
$L\times_Y\bar\pp_{\D,\mm}$
for some line bundle $L$ on $Y$. By Lemma \ref{bij} this 
$\Aut(\D, \M)$-torsor structure extends to the structure 
of a transitive \hc torsor over $\la\Aut(\D,\mm), \Der(\D,\mm)\ra$
if and only if 
$ c_1(L) = \alpha(\bar\pp_{\D,\mm}, c_{\Der}, \io_{\Der})$
in $H^2(Y, \Om^{\geq 1}_Y)$. 

From Proposition \ref{jdm} the latter class is also equal to 
$\alpha(\bar\Phi_*\bar\pp_{\D,\mm}, c_\J, \io_\J)+\half c_1(K_Y)$. 
By equivariant descent with respect to $(1 + \hb K)$ and 
Lemma \ref{alat} we identify the first of the two terms as
$At(\oo_\hb, Y)$. 

We conclude
that the equation $c_1(L)=\half c_1(K_Y)+
At(\oo_\hb,Y)$ holds if and only if  the $\la\oa,\od\ra$-torsor
$\bar\pp_{\D,\mm}$ can be lifted to a torsor
$\pp_{\D,\mm}$ as in \eqref{barpp}.
\end{proof}

\bigskip
We now discuss the set of isomorphism classes
of the lifts $\pp_{\D,\mm}$ for a fixed 
$\overline{\pp}_\J$ to , assuming it is non-empty. Assume that
$L$ is a flat $(\mathcal{O}_Y[[\hb]])^\times$ torsor (but we do not 
fix a choice of a flat connection), then 
$\pp_{\D,\mm}(L) := 
(\pp_{\D,\mm} \times_Y L)/ K^\times$ is again a  torsor over
$\Aut(\D, \M)$, since $K^\times$ is central in $\Aut(\D, \M)$. 
The fact that this extends to a transitive \hc torsor structure 
on $\pp_{\D,\mm}(L)$ can be established as follows. 
We have a direct product decomposition
 $(\mathcal{O}_Y[[\hb]])^\times\ \cong\ \oo_Y^\times \times 
(1 + \hb \oo_Y)$. Therefore, choosing $L$ is equivalent to choosing
a pair $(L_0, L_1)$ consisting of a flat $\oo_Y^\times$-torsor $L_0$
and a flat $(1 + \hb \oo_Y)$-torsor $L_1$. 

It is clear that
the lift $\bar\pp_{\D,\mm}$ can be adjusted by passing to 
$\bar\pp_{\D,\mm}(L_1)$ (which is defined similarly), since $L_1$ has
a trivial class in $\hb H^2_\dr(Y)[[\hb]]$. 
By Proposition 2.7 in \cite{BK}, any lift of $\pp_\J$ to 
 a $\la\oa,\od\ra$-torsor has the form $\bar\pp_{\D,\mm}(L_1)$
for a unique $L_1$.
Similarly, with fixed choice of $\bar\pp_{\D,\mm}$, any two  lifts to a 
$\la\Aut(\D,\mm), \Der(\D,\mm)\ra$-torsor
$\pp_{\D,\mm}$ differ by a twist by a unique 
$L_0$, as follows from Lemma
\ref{bij}. Hence, every lift of $\pp_\J$ is isomorphic to 
$\pp_{\D,\mm}(L)$ for a unique flat $(\mathcal{O}_Y[[\hb]])^\times$-torsor
$L$, as required. This finishes the proof of Theorem 
\ref{maintheorem}. \qed

\subsection{Final remarks}

Given a line bundle $L$, one $Y$,
one can consider a problem of quantization of $L$
up to a finite order in $\hb$. That is, for
any $s=1,2,\ldots$, one can
study deformations of $L$ to an
$\mathcal{O}_\hb / \hb^{s+1} \mathcal{O}_\hb $-module $L_s$, which is 
flat over $\k[\hb]/\hb^{s+1}$. 
The corresponding `finite order'
counterpart of Theorem \ref{maintheorem} is more complicated, in a sense,
than Theorem \ref{maintheorem} itself.

To explain this, for each $N\geq 1$, consider the following condition
\[(\star_N):\quad
c_1(L) - \frac{1}{2} c_1(K_Y) =At(\ohb,Y)
\en\ \&\en\ i_Y^* \omega_i(\ohb)=0
\en\forall i=1,\ldots,N.\]
Then, it turns out that one has  implications
\[(\star_{s+1})
\quad\Rightarrow\quad
\exists\, L_s
\quad\Rightarrow\quad
(\star_{s-1}),
\]
however, none of the two implications is an equivalence, in general.
The origin of this phenomenon comes from the
fact that the classes in the sequence $\omega_i(\ohb),\ i=1,2,
\ldots$, are, essentially, the obstructions to extending
a torsor over the \hc pair $\la\G/\G^{\geq i},\GG/\GG^{\geq i}\ra$
to a torsor over  $\la\G/\G^{\geq i+1},\GG/\GG^{\geq i+1}\ra$.
Thus, the  sequence  $\omega_i(\ohb),\ i=1,2,
\ldots$, is closely related to the descending filtration
$D^{\geq i}$ on the algebra $D$.
On the other hand, associated with the choice of a Lagrangian subspace
$\y$ there is   another
descending filtration, $F_\y^{\geq i}D$, on $D$. It is defined  as the
multiplicative filtration on the enveloping algebra of the Heisenberg
algebra $\v\oplus\k\hb$ induced by the 3-step 
filtration 
\[F^0=\v\oplus\k\hb\ \supset\ F^1=\v\ \supset\ F^2=\y,\]
 on the vector space  $\v\oplus\k\hb$.
The obstructions for the existence of
finite order deformations of the line bundle $L$
are more naturally related to the filtration $F_\y^{\geq i}D$
rather than to $D^{\geq i}$.

We illustrate the above in the case $s=1$.
Assume for simplicity that
$Y$ is (smooth)  projective and that the sheaf
 $\mathcal{O}_\hb / \hb^{2} \mathcal{O}_\hb $
splits globally as $\mathcal{O}_X + \hb \mathcal{O}_X$. Then,
by Hodge theory, the cohomology group $H^2(\Om_Y^{\geq1})$ is a subspace
of $H^2_\dr (Y)$ and we have $At(\oo_{\hb}, Y) = i^*_Y \omega_1(\oo_\hb)$. 

By \cite{BGP}, a first order deformation exists Zariski 
locally on $X$ if and only if $\omega|_Y = 0$, which corresponds to 
$i =0$ and a vanishing in the $H^0(Y, \Omega_Y^2)$ component of
$H^2_\dr (Y)$.  The local first order deformations can be adjusted 
so that they are isomorphic on the double intersections of 
our Zariski covering, if and only if
$2 c_1(L) = c_1(K_Y)$ in $H^1(Y, \Omega_Y^1)$, which is 
a part of our equation on $i^*_Y \omega_1(\oo_\hb)$ (since it is 
represented by a closed 2-form on $Y$ under our assumptions
and the projection onto $H^1(Y, \Omega^1_Y)$ is zero). 

Further, the 
isomorphisms on double intersections can be chosen to 
satisfy the cocycle condition on triple intersections precisely 
when a certain class in $H^2(Y, \mathcal{O}_Y)$ vanishes. 
In \cite{BGP} we were unable to evaluate this class explicitly, 
but very recently a formula for it was found by K. Behrend and
B. Fantechi and we expect that the corresponding vanishing
condition is the $H^2(Y, \mathcal{O}_Y)$ component of
$i^*_Y \omega_2(\oo_\hb) = 0$. 

Finally, by \cite{BGP}, the $H^0(Y, \Omega_Y^2)$ component of 
the equation for $i^*_Y \omega_1(\oo_\hb)$ is precisely the condition that
$L$ admits a local second order deformation. Again, local 
second order deformations can be globalized when the projections
of $i^*_Y \om_2(\oo_\hb)$ to $H^1(Y, \Om^1_Y)$ and $i^*_Y \om_3(\oo_\hb)$
to $H^2(Y, \oo_Y)$ vanish, respectively, and so on.

\bibliographystyle{plain}

\end{document}